\documentclass[12pt,a4paper,twoside,english,american]{scrartcl}
\usepackage{helvet}
\usepackage[T1]{fontenc}
\usepackage[latin9]{inputenc}
\usepackage{amsmath}
\usepackage{setspace}
\usepackage{amssymb}
\usepackage{enumerate}

\makeatletter

 \usepackage{amsthm}
 \theoremstyle{plain}
\newtheorem{thm}{Theorem}[section]
  \theoremstyle{plain}
  \newtheorem*{thm*}{Theorem}
  \theoremstyle{plain}
  \newtheorem{cor}[thm]{Corollary}
  \theoremstyle{remark}
  \newtheorem{rem}[thm]{Remark}
  \theoremstyle{definition}
  \newtheorem{defn}[thm]{Definition}
 \theoremstyle{definition}
 \newtheorem*{defn*}{Definition}
  \theoremstyle{plain}
  \newtheorem{lem}[thm]{Lemma}
 \theoremstyle{definition}
  \newtheorem{example}[thm]{Example}
  \theoremstyle{remark}

\usepackage[T1]{fontenc}    

\usepackage{a4wide}        
\usepackage{tu-preprint}
\usepackage{amsmath}

\usepackage{euscript}
\usepackage{mathabx}
\allowdisplaybreaks[1] 
\usepackage{amsfonts}
\newenvironment{keywords}{ \noindent\footnotesize\textbf{Keywords and phrases:}}{}

\newenvironment{class}{\noindent\footnotesize\textbf{Mathematics subject classification 2010:}}{}

\usepackage{color,tu-preprint}

\newcommand{\Abs}[1]{\left\lVert#1\right\rVert}

\newcommand{\R}{\mathbb{R}}

\newcommand*{\abs}[1]{\lvert#1\rvert}

\newcommand*{\sgn}{\operatorname{sgn}}
\newcommand*{\supp}{\operatorname{supp}}

\newcommand{\s}[1]{\mathcal{#1}}

\DeclareMathAccent{\Circ}{\mathalpha}{operators}{"17}
\newcommand{\interior}[1]{\Circ{#1}}

\newcommand{\lspan}{\operatorname{span}}

\renewcommand{\hat}{\widehat}
\renewcommand{\tilde}{\widetilde}
\renewcommand*{\epsilon}{\varepsilon}
\renewcommand*{\rho}{\varrho}

\makeatother

\usepackage{babel}

\begin{document}
\selectlanguage{english}%
\institut{Institut f\"ur Analysis}

\preprintnumber{MATH-AN-015-2012}

\preprinttitle{A Hilbert Space Perspective on Ordinary Differential Equations\\ with Memory Term.}

\author{Anke Kalauch, Rainer Picard, Stefan Siegmund, Sascha Trostorff \& Marcus Waurick}

\makepreprinttitlepage

\selectlanguage{american}%


\title{A Hilbert Space Perspective on Ordinary Differential Equations\\ 
with Memory Term}

\author{Anke Kalauch, Rainer Picard, Stefan Siegmund,\\
 Sascha Trostorff \& Marcus Waurick
\\[0.5ex]
 Institut f\"ur Analysis, Fachrichtung Mathematik\\
 Technische Universit\"at Dresden\\
 Germany
\\[0.5ex]
 anke.kalauch@tu-dreden.de\\
 rainer.picard@tu-dresden.de\\
 stefan.siegmund@tu-dresden.de\\
 sascha.trostorff@tu-dresden.de\\
 marcus.waurick@tu-dresden.de }
\maketitle
\begin{abstract}
\textbf{Abstract.} We discuss ordinary differential equations with delay and memory terms in Hilbert spaces. By introducing a time derivative as a normal operator in an appropriate Hilbert space, we develop a new approach to a solution theory covering integro-differential equations, neutral differential equations and general delay differential equations within a unified framework. We show that reasonable differential equations lead to causal solution operators.
\end{abstract}


\begin{class} 34K05 (Functional differential equations, general theory); 34K30 (Equations in abstract spaces); 34K40 (Neutral equations); 34A12 (Initial value problems, existence, uniqueness, continuous dependence and continuation of solutions); 45G15 (Systems of nonlinear integral equations)
\end{class}

\tableofcontents{}

\section{Introduction}

Delay differential equations (DDE)
\begin{equation}\label{eq:delay}
  \dot x(t) = f(t,x_t)
\end{equation}
are important in many areas of engineering and science (see e.g.\ \cite{Bal2009} and the references therein).
Let $t_0,t_1\in \mathbb{R}$ with $t_0<t_1$. If $a \in \R_{\geq 0}$ denotes the \emph{maximal delay} and $x : [t_0-a,t_1] \rightarrow \R^n$ is a continuous function then for $t \in [t_0,t_1]$ the function $x_t : [-a,0] \rightarrow \R^n$ denotes its (restricted) \emph{translate}, which is defined by $x_t(\theta) = x(t + \theta)$ for $\theta \in [-a,0]$. A continuous function $x : [t_0-a,t_1] \rightarrow \R^n$ is a \emph{solution of \eqref{eq:delay}} which takes at the \emph{initial time} $t_0$ the \emph{initial value} $\phi \in C := C([-a,0],\R^n)$, if it satisfies the following initial condition and integral equation \cite[Lemma 1.1]{HalLun1993}
\begin{equation}\label{eq:integral}
  x_{t_0} = \phi, \quad x(t) = \phi(0) + \int_{t_0}^t f(s,x_s) \,ds
  \qquad \text{for } t \in [t_0,t_1] .
\end{equation}
Note that it is assumed that $f : D \rightarrow \R^n$ is defined on an open subset $D \subseteq \R \times C$ and $[t_0,t_1] \ni s \mapsto f(s,x_s) \in \R^n$ is well-defined and continuous.

Instead of viewing solutions $x : [t_0-a,t_1] \rightarrow \R^n$ as functions with values in $\R^n$ one can equivalently consider the induced mapping $[t_0,t_1] \ni t \mapsto x_t \in C$ (see e.g.\ \cite{DieGilVerWal,HalLun1993}). Utilizing \eqref{eq:integral} and a contraction argument one can show the following existence theorem.

\begin{thm}[{\cite[Theorem 2.3]{HalLun1993}}]
Suppose that $f : D \rightarrow \R^n$ is continuous and Lipschitz continuous w.r.t.\ the second argument on every compact subset of $D$. Then for $(t_0,\phi) \in D$ equation \eqref{eq:delay} subject to the initial condition $x_{t_0} = \phi$, has a unique solution.
\end{thm}

Examples of delay differential equations \eqref{eq:delay} include
\begin{itemize}
  \item ordinary differential equations ($a=0$) $\dot x(t) = F(x(t))$
  \item differential difference equations $\dot x(t) = f(t,x(t),x(t-\tau_1(t)),\dots,x(t-\tau_p(t)))$ with $0 \leq \tau_j(t) \leq a$, $j=1,\dots,p$
  \item integro-differential equations $\dot x(t) = \int_{-a}^0 g(t,\theta,x(t+\theta)) \,d\theta$
\end{itemize}
Another important class of delay differential equations which generalizes \eqref{eq:delay} are
\begin{itemize}
  \item neutral delay differential equations $\dot x(t) = f(t,x_t,\dot x_t)$, i.e., \ $\dot x(t)$ depends not only on the past history of $x$ but also on the past history of $\dot x$
\end{itemize}

For an approach to a general class of partial differential delay equations, see \cite{Ruess0, Ruess1}. 


Solution theories for DDEs \eqref{eq:delay} adopt the point of view of semigroup theory and have as a common feature that they work with the integral equation representation \eqref{eq:integral} (\cite{DieGilVerWal, DiekGyl2008, Doan11}). In this paper we suggest a new approach to delay equations which is based on a unified Hilbert space approach for differential equations by Picard et al.\ \cite{0753.44002, Pi2009-1, PicMcG2011} and focusses on the properties of the time derivative as an operator on tailor-made solution spaces of exponentially bounded functions on the whole real line. Applying this approach to DDEs, it turns out that unbounded delay is not more difficult than bounded delay and many different classes of equations with memory, including integro-differential and neutral delay equations, can be treated in a unified way. The method consists in solving the differential equation in a weak sense. As a trade-off, it is possible to solve equations with source terms like characteristic functions or finite measures with bounded support on $\R$.

In order to illustrate one of the main ideas of this new approach, namely the time derivative as an operator acting on a function space of potential solutions defined on the whole real line, consider a simple special case of DDEs: scalar autonomous ordinary differential equations and their corresponding initial value problems
\[
  \dot x(t) = f(x(t)) \textup{ for } t \in \R_{> 0}, \qquad x(0) = x_0 \in \R
  .
\]
For simplicity assume that $f : \R \rightarrow \R$ is Lipschitz continuous and satisfies $f(0) = 0$. To extend this initial value problem on $\R_{\geq 0}$ to $\R$, denote by $\tilde x : \R_{\geq 0} \rightarrow \R$ its unique solution. $\tilde x$ is continuously differentiable and $\tilde x(0+) = x_0$. We trivially extend $\tilde x$ by setting
\[
  x : \R \rightarrow \R,\;
  t \mapsto
  \left\{
    \begin{array}{ll}
      0, & \textup{ for } t \in \R_{<0},
      \\
      \tilde x(t),\; & \textup{ for } t \in \R_{\geq 0}.
    \end{array}
  \right.
\]
The extension $x$ is continuously differentiable on $\R\setminus\{0\}$. Utilizing the Heavyside function $\chi_{\R_{>0}}$, we compute the distributional derivative of $x$
\[
  x' = (\chi_{\R_{>0}} \cdot x)' = x(0+)\delta + \chi_{\R_{>0}} x' =
       x_0 \delta + \chi_{\R_{>0}} x'
\]
where $\delta$ is the Dirac delta distribution. Hence, the distributional derivative of $x$ satisfies (keep in mind that $f(0) = 0$)
\[
  x' = f(x(\cdot)) + x_0 \delta
  .
\]
With this example in mind, we develop a solution theory for right hand sides being distributions. Moreover, we will cover a comprehensive class of initial value problems within our perspective (Theorem \ref{thm:ivode}). The simplifying assumption $f(0) = 0$ will be replaced by the notion of \emph{causality} (Theorem \ref{thm:causality}). In order to consistently develop an operator-theoretic point of view, we need a glimpse on extrapolation spaces and view the time derivative as an operator on embedded Sobolev spaces of exponentially weighted functions on $\R$ which form a Gelfand triple.

The structure of the paper is as follows. In Section \ref{sec:TTD} we introduce the time-derivative as a continuously invertible operator in an $L^2$-type space. We define the time-derivative on the whole real line in order to be able to use the Fourier transform and have a functional calculus at hand which allows to consider autonomous linear equations of neutral type very easily (Theorem \ref{thm:neutral}). We keep the paper self-contained but cite the basic concepts from \cite{PicMcG2011}.

Section \ref{sec:BST} gives the basic solution theory for differential equations of the form $\dot u = F(u)$, where $F$ is Lipschitz continuous in some appropriate sense. Here, the term solution theory means that we define a suitable Hilbert space in order to establish existence, uniqueness and continuous dependence on the data. It should be noted that the respective proofs are based on the contraction mapping theorem and are therefore comparatively elementary. We emphasize that since our approach is well-suited for the Hilbert space setting, there are very few technical difficulties. In particular, we do not encounter problems due to the non-reflexivity of the underlying space. This may lead to an easier treatment of stability and bifurcation analysis, since it avoids using the sun-star calculus \cite{DieGilVerWal}.

Section \ref{Sec: CausMem} is devoted to the study of causality which roughly means that a solution up to time $t$ only depends on the right hand side up to time $t$. We also state a notion which is dual to causality and which may be called amnesic. This concept gives a possible way to rigorously \emph{define} what it means for a differential equation to be a delay differential equation. The notion of causality is of prime importance to describe physically reasonable equations, it was so far almost not present in the classical literature on differential equations and only recently gained some attention \cite{Laksh10,Thomas1997,Weiss2003}. Section \ref{Sec: CausMem} also provides a theorem (Theorem \ref{thm:causality}) that states that reasonable right hand sides lead to causal solution operators.

Section \ref{sec:app} illustrates the applicability and versatility of the concepts developed in the Sections \ref{sec:BST} and \ref{Sec: CausMem}. The first part of this section deals with the formulation of initial value problems within our context. In the second part we present a method to establish local solvability. We only give an introductory example of how to treat equations of the type $\dot u(t) = g(t,u(t))$, where we assume $g$ to be only locally Lipschitz continuous. Adopting this strategy to the general case would then also lead to a respective local solution theory. However, since we aim to cover a broad class of a priori different structures, we focus on global-in-time solutions in order to keep this exposition rather elementary and to avoid many technicalities.
In the third part of this section we study special types of delay differential equations including some retarded functional differential equations, integro-differential equations and equations of neutral type.

To fix notation, let $H$ denote a Hilbert space over the field $\mathbb{K}$  where $\mathbb{K}\in \{\R,\mathbb{C}\}$.


\section{The Time Derivative}\label{sec:TTD}

To develop our operator-theoretic approach to differential equations with memory, it is pivotal to establish
the time derivative as a boundedly invertible operator in an adequate Hilbert
space setting. This strategy uses specific properties of the one-dimensional derivative on the real line.
In order to describe these structural features appropriately, we need the following operators.

\begin{defn} Denoting by $\interior C_{\infty}\left(\mathbb{R}\right)$ the set of smooth functions with compact support and by $L^2(\mathbb R)$ the Hilbert space (of equivalence classes) of square-integrable functions w.r.t. Lebesgue measure, we define the time derivative
\[
\partial_{c}:\interior C_{\infty}\left(\mathbb{R}\right)\subseteq L^{2}(\mathbb{R})\to L^{2}(\R):\phi\mapsto\phi'\]
 and the multiplication operator 
\[
m_{c}:\interior C_{\infty}\left(\mathbb{R}\right)\subseteq L^{2}(\R)\to L^{2}(\R):\phi\mapsto(x\mapsto x\phi(x)).\]
 Moreover, the Fourier transform $\s F_c$ is defined in the following way
\begin{align*}
\s F_{c}: & \interior C_{\infty}\left(\mathbb{R}\right)\subseteq L^{2}(\R)\to L^{2}(\R)\\
 & \phi\mapsto\left(x\mapsto\frac{1}{\sqrt{2\pi}}\int_{\R}e^{-\mathrm{i}xy}\phi(y)dy\right).\end{align*}
 \end{defn}

\begin{thm}\label{thm: operators} The operators $\partial_{c}$, $m_{c}$, $\s F_{c}$ are densely
defined and closable. Moreover, for $\partial:=-\partial_{c}^{*},m:=\overline{m_{c}},\s F:=\overline{\s F_{c}}$
the following properties hold
\begin{enumerate}[(a)]
\item \label{ab3} $\s F$ is unitary from $L^{2}(\R)$ onto $L^{2}(\R)$,

\item \label{ab1} $D(\partial)=\{f\in L^{2}(\R);f'\in L^{2}(\R)\}$, where $f'$ denotes the distributional derivative,

\item \label{ab2} $D(m)=\{f\in L^{2}(\R);(x\mapsto xf(x))\in L^{2}(\R)\}$,

\item \label{ab4} $\partial,\mathrm{i}m$ are skew-selfadjoint,

\item \label{ab5} $\partial=\s F^{*}\mathrm{i}m\s F$. \end{enumerate}
\end{thm} \begin{proof} The operators $\partial_{c}$, $m_{c}$, $\s F_{c}$
are clearly densely defined. Closability of $\s F_{c}$ is clear,
since $\s F_{c}$ is norm-preserving and has dense range \cite[Theorem V.2.8 and  Lemma V.1.10]{Wer2007}.
Consequently, \eqref{ab3} follows. The closability of $\partial_{c}$ and $m_{c}$
follows from $\partial_{c}\subseteq\partial$ and $m_{c}\subseteq m^{*}$,
respectively. \eqref{ab1} is immediate from the definition of distributional
derivatives and \eqref{ab2} is easy. The fact that $\mathrm{i}\partial$ is selfadjoint can be found in \cite[Example 3.14]{Kat1980}. \eqref{ab5} is proved in \cite[Volume 1, p.161-163]{0467.47001}. Thus, also \eqref{ab4} follows. \end{proof}

Before we turn to further studies on the operator $\partial$, we
need to define weighted $L^{2}$-type spaces. This was also done in
\cite[Section 1.2]{0753.44002} and for convenience and to fix some
notation it will be paraphrased shortly.
\begin{defn}\label{Def: weightedL2} Let $\rho \in \R$.
Define $H_{\rho,0}(\R):=\{f\in L_{\textnormal{loc}}^{2}(\R);(x\mapsto\exp(-\rho x)f(x))\in L^{2}(\R)\}$.
We endow $H_{\rho,0}(\R)$ with the scalar product \[
(f,g)\mapsto\langle f,g\rangle_{\rho,0}:=\int_{\R}f(x)^{*}g(x)\exp(-2\rho x)dx.\]
 We note that the associated norm, denoted by $\left|\:\cdot\:\right|_{\rho,0}$,
is an $L^{2}$-type variant of the ``Morgenstern-norm'' (\cite{Ref195229}), which is familiar from the proof of the classical Picard-Lendelöf theorem. Moreover, we define the following unitary operator \[
\exp(-\rho m):H_{\rho,0}(\R)\to L^{2}(\R):f\mapsto(x\mapsto\exp(-\rho x)f(x)).\]
By the unitarity of $\exp(-\rho m)$ we have $\exp(-\rho m)^{-1}=\exp(-\rho m)^*$.
\end{defn}

\begin{rem}\label{rem: selfadjoint} We note that for selfadjoint operators $A$ the spectrum is purely
real, i.e., $\sigma(A)\subseteq\R$. Consequently, the spectrum of
skew-selfadjoint operators is contained in the imaginary axis.
\end{rem}

Denoting by $\lVert\cdot\rVert_{L(X,Y)}$ the operator norm of a bounded linear operator from the Banach space $X$ to the Banach space $Y$, we have the following corollary.

\begin{cor}\label{cor:FLT} Let $\rho\in \R\setminus\{0\}$. Define $\partial_{\rho}:=\exp(-\rho m)^{-1}\partial\exp(-\rho m)$.
Introducing the \emph{Fourier-Laplace transform} $\s L_{\rho}:=\s F\exp(-\rho m)$
and $\partial_{0,\rho}:=\partial_{\rho}+\rho$, we have the following
\begin{enumerate}[(a)]

\item \label{cor1} $\partial_{\rho}=\s L_{\rho}^{*}\mathrm{i}m\s L_{\rho}$,

\item \label{cor2} the operators $\mathrm{i}m+\rho$, $\partial_{0,\rho}$ are continuously invertible,

\item \label{cor3} $\partial_{0,\rho}^{-1}=\s L_{\rho}^{*}(\mathrm{i}m+\rho)^{-1}\s L_{\rho}$,

\item \label{cor4} $\lVert \partial_{0,\rho}^{-1} \rVert_{L(H_{\rho,0}(\R),H_{\rho,0}(\R))} = \lVert (\mathrm{i}m +\rho)^{-1} \rVert_{L(L^{2}(\R),L^{2}(\R))} = \frac1{\abs{\rho}}$.
\end{enumerate} Moreover, the following formula holds \[
\left(\partial_{0,\rho}^{-1}\varphi\right)\left(x\right)=\left(\left(\partial_{\rho}+\rho\right)^{-1}\varphi\right)\left(x\right)=\begin{cases} \int_{-\infty}^{x}\varphi\left(t\right)\: dt, & \rho>0,\\
-\int_{x}^{-\infty}\varphi\left(t\right)\: dt, & \rho<0,\end{cases}
\]
 for all $\varphi\in\interior C_{\infty}\left(\mathbb{R}\right)$
and $x\in\mathbb{R}$. \end{cor} 
\begin{proof} \eqref{cor1} is
immediate from Definition \ref{Def: weightedL2} and Theorem \ref{thm: operators}. \eqref{cor2}
follows from Remark \ref{rem: selfadjoint}. \eqref{cor3}
is clear by \eqref{cor1}. The remaining formulas are elementary. \end{proof}

We emphasize that $\partial_{0,\rho}$ only depends on $\rho$ with regards to the domain of definition, i.e., for all $f\in D(\partial_{0,\rho_1})\cap D(\partial_{0,\rho_2})$ it holds $\partial_{0,\rho_1}f=\partial_{0,\rho_2}f$ for all $\rho_1,\rho_2\in \R$. Moreover, $\interior C_{\infty}\left(\mathbb{R}\right)$ is a core for $\partial_{0,\rho}$ and we have for $\phi\in\interior C_{\infty}\left(\mathbb{R}\right)$,
the equality
 \[
    \partial\phi=\partial_{0,\rho}\phi=(\partial_{\rho}+\rho)\phi.
 \]
We will show that our solution theory is independent of  $\rho$ (see Remark \ref{Rem: rho-Independent} and Theorem \ref{thm: rho-Independent} for a more detailed discussion).

In our approach to delay differential equations, we will need a glimpse
on extrapolation spaces. The core of the issues needed here is summarized
in the next definition and the subsequent theorem.

\begin{defn} For $\rho\in \R\setminus\{0\}$ we define $H_{\rho,-1}(\mathbb{R}):= H_{-\rho,1}(\mathbb{R})^*$ the space of continuous linear functionals on $H_{-\rho,1}(\R)$. The respective norm is denoted by $\abs{\cdot}_{\rho,-1}$. Moreover,
define $H_{\rho,1}(\mathbb{R})$ as the Hilbert space $D(\partial_{0,\rho})$
endowed with the norm $\abs{\cdot}_{\rho,1}:\phi\mapsto\abs{\partial_{0,\rho}\phi}_{\rho,0}$.
\end{defn}

\begin{rem}\label{rem: Identification}
  We note that we may identify $H_{\rho,0}(\R) \subseteq H_{\rho,-1}(\R)$ via the mapping
   \[
      H_{\rho,0}(\R) \ni \phi \mapsto \left(H_{-\rho,1}(\R) \ni \psi \mapsto \langle \exp(-\rho m) \phi, \exp(\rho m) \psi \rangle_{0,0} =: \langle \phi,\psi\rangle_{0,0} \in \mathbb{K}\right) \in H_{\rho,-1}(\R).
   \]
  We shall do so henceforth.
\end{rem}

\begin{thm}{\label{thm:sobolev chain}} Let $\rho\in \R\setminus\{0\}$. Then we have the following chain of continuous
embeddings \[
H_{-\rho,1}(\mathbb{R})\hookrightarrow H_{-\rho,0}(\mathbb{R})\stackrel{\exp(\rho m)}{\rightarrow}  H_{0,0}(\mathbb{R}) \stackrel{\exp(-\rho m)^{-1}}{\rightarrow} H_{\rho,0}(\mathbb{R}) \hookrightarrow H_{\rho,-1}(\mathbb{R}),\]
 the triple $(H_{-\rho,1}(\mathbb{R}),H_{0,0}(\mathbb{R}),H_{\rho,-1}(\mathbb{R}))$
is a \emph{Gelfand triple}. The following mappings \begin{align*}
\partial_{0,1\to0}: & H_{-\rho,1}(\mathbb{R})\to H_{-\rho,0}(\mathbb{R})\\
 & \phi\mapsto\partial_{0,-\rho}\phi\\
\partial_{0,0\to1}: & H_{-\rho,0}(\mathbb{R})\to H_{-\rho,1}(\mathbb{R})\\
 & \phi\mapsto\partial_{0,-\rho}^{-1}\phi \\
\partial_{0,0\to-1}: & H_{\rho  ,0}(\mathbb{R})\to H_{\rho,-1}(\mathbb{R})\\
 & \phi\mapsto(H_{-\rho,1}(\R)\ni\psi \mapsto \langle \phi,- \partial_{0,-\rho} \psi \rangle_{0,0})\\
\partial_{0,-1\to0}: & H_{\rho,-1}(\mathbb{R})\to H_{\rho,0}(\mathbb{R})\\
 & \phi\mapsto\left( H_{-\rho,0}(\R)\ni\psi \mapsto \phi(-\partial_{0,-\rho}^{-1}\psi)\right) \end{align*}
 are unitary. \end{thm}
\begin{proof} This is part of the more general construction of Sobolev chains and
can be found in \cite{PicMcG2011, Wau2011}. \end{proof}
\newpage
\begin{rem}\label{rem: dualdescr}~

 \begin{enumerate}[(a)]
\item For the sake of simplicity, we will write $\partial_{0,-\rho}$ for $\partial_{0,1\to0}$, $\partial_{0,-\rho}^{-1}$ for $\partial_{0,0\to1}$, $\partial_{0,\rho}$ for $\partial_{0,0\to-1}$ and $\partial_{0,-\rho}$ for $\partial_{0,-1\to0}$. With this notation, we arrive at the following duality. For $\phi,\psi \in \interior C_\infty(\R)$ and $\rho\in\R\setminus\{0\}$ we have
\[
    \langle \partial_{0,\rho} \phi , \psi \rangle_{0,0} = \langle \phi, -\partial_{0,-\rho}\psi \rangle_{0,0} 
    \quad \text{and} \quad
    \langle \partial_{0,\rho}^{-1} \phi , \psi \rangle_{0,0} = \langle \phi, -\partial_{0,-\rho}^{-1}\psi \rangle_{0,0}.
\]
\item\label{rem: dualdescr2}  We shall mention another possible duality, namely that of the Hilbert space adjoint, which in this setting may be computed. It is straightforward to show that the adjoint $(\partial_{0,\rho}^{-1})^*$ of $\partial_{0,\rho}^{-1} : H_{\rho,0}(\R) \to H_{\rho,0}(\R)$ is given by
\[
    (\partial_{0,\rho}^{-1})^* = \exp(\rho m)^{-1} \exp(-\rho m) \partial_{0,-\rho}^{-1} \exp(-\rho m)^{-1} \exp(\rho m).
\]
We may also use the continuous extension of $(\partial_{0,\rho}^{-1})^*$ as a Banach space isomorphism from $H_{\rho,-1}(\R)$ to $H_{\rho,0}(\R)$.
 \item  Assume $\rho>0$. Developing our approach to delay equations,
we will define what it means for a mapping to have delay or to be
amnesic in Section \ref{Subsec: DelayMem}. Prototypes for the former are
$\partial_{0,\rho}^{-1}$ and $(\partial_{0,-\rho}^{-1})^*$,  for the latter $(\partial_{0,\rho}^{-1})^*$ and $\partial_{0,-\rho}^{-1}$.

\item All the theory developed above can be generalized to function
spaces with values in a Hilbert space. A way for doing so is by means
of tensor product constructions, which can be found in \cite{Wei1980, PicMcG2011}. We use the notation introduced in \cite{PicMcG2011}, i.e., for two Hilbert spaces $H_1$ and $H_2$ the tensor product will be denoted by $H_1\otimes H_2$. For total subsets $D_1\subseteq H_1$ and $D_2\subseteq H_2$, we denote the algebraic tensor product by
\[
  D_1\stackrel a\otimes D_2 :=\lspan \{ \phi \otimes \psi ; \phi \in D_1 ,\psi \in D_2\}.
\]
Let $H$ be a Hilbert space and $A: D(A)\subseteq H_1 \to H_2$ be a closable, densely defined linear operator. The identity in $H$ is denoted by $1_H$. Let $D\subseteq H$ be dense. For $\phi\in D(A)$ and $\psi\in D$, we may define
\[
   (A\stackrel a\otimes 1_H)(\phi\otimes \psi) := A\phi \otimes \psi.
\] It can be shown that $A\stackrel a\otimes 1_H$ has a well-defined linear extension as an operator from $D(A)\stackrel a\otimes D\subseteq H_1 \otimes H$ to $H_2 \otimes H$, we re-use the name $A\stackrel a\otimes 1_H$ for that extension. The closure of $A\stackrel a\otimes 1_H$ is again a linear operator, which will be denoted by $A \otimes 1_H$. An analogous construction can be done for $1_H\otimes A$.

\item We specialize the latter remark. If $H$ is a Hilbert space of functions, e.g., if $H= L^2(\Omega,\Sigma,\mu)$ for some measure space $(\Omega,\Sigma,\mu)$ the Hilbert space tensor product \linebreak $L^2(\Omega,\Sigma,\mu)\otimes H_1$ is isometrically isomorphic to the space of square-integrable $H_1$-valued $L^2$-functions $L^2(\Omega,\Sigma,\mu;H)$. In that case the operator $1_{L^2(\Omega,\Sigma,\mu)}\otimes A$ is the canonical extension of $A$, which acts as
\[
  (1_{L^2(\Omega,\Sigma,\mu)}\otimes A)(t\mapsto \chi_S(t)\phi) = (t\mapsto \chi_S(t)A\phi),
\]
where $S\in \Sigma, \mu(S)<\infty, \phi \in D(A)$ and $\chi_S$ denotes the characteristic function of the set $S$. We shall use the identification of the tensor product and the space of vector-valued functions in the sequel. We will not distinguish notationally between the norm in $H_{\rho,k}(\R)$ and $H_{\rho,k}(\R)\otimes H$ for $k\in \{-1,0,1\}$. Moreover, we will also use the name $A$ for the extension of an operator to the tensor product.
\end{enumerate} \end{rem}

The operator $\partial_{0,\rho}^{-1}$ is a normal operator in $H_{\rho,0}(\R)$. Therefore, it admits a functional calculus, which may be generalized to operator-valued functions. Note that, similar to the forumlas for $\partial_{0,\rho}^{-1}$, the functional calculus depends on the sign of $\rho$, which we denote by $\sgn(\rho)$. We denote by $B_\mathbb{C}(x,r)$ the open ball in $\mathbb{C}$ around $x\in \mathbb{C}$ with radius $r \in \R_{>0}$.

\begin{defn}\label{def:funcalc} Let $\rho \in \R\setminus\{0\}$. Define $\mathbb{L}_{\rho}:=\mathcal{L}_{\rho}\otimes 1_{H}$.
Let $r>\frac{1}{2\abs{\rho}}$ and $M:B_{\mathbb{C}}(\sgn(\rho)r,r)\to L(H)$ be bounded
and analytic. Define \[
M\left(\partial_{0,\rho}^{-1}\right):=\mathbb{L}_{\rho}^{*}\: M\left(\frac{1}{\mathrm{i}  m+\rho}\right)\:\mathbb{L}_{\rho},\]
 where \[
M\left(\frac{1}{\mathrm{i}   m+\rho}\right)\phi(t):=M\left(\frac{1}{\mathrm{i}  t+\rho}\right)\phi(t)\quad(t\in\mathbb{R})\]
 for $\phi\in\interior C_{\infty}\left(\mathbb{R};H\right)$. \end{defn}
\vspace{0.5cm}
\begin{rem}\label{Rem: rho-Independent}~

\begin{enumerate}[(a)]
\item\label{Rem: rho-Independent0} It is easy to see that $\lVert M(\partial_{0,\rho}^{-1})\rVert_{L(H_{\rho,0}(\R)\otimes H, H_{\rho,0}(\R)\otimes H)}\leq \sup_{z\in B(\sgn(\rho)r,r)}\lVert M(z)\rVert_{L(H)}$.
\item\label{Rem: rho-Independent1} The definition of $M(\partial_{0,\rho}^{-1})$ is largely independent of
the choice of $\rho$ in the sense that the operators $M(\partial_{0,\rho_1}^{-1})$ and $M(\partial_{0,\rho_2}^{-1})$ for two different
parameters $\rho_{1},\rho_{2}$ such that $\rho_1\cdot\rho_2>0$ coincide on the intersection of the
respective domains. A proof of this statement can be found in \cite[Theorem 6.1.4]{PicMcG2011}
or \cite[Lemma 1.3.2]{Wau2011}.
\item\label{Rem: rho-Independent2} The operator $M(\partial_{0,\rho}^{-1})$ has a unique continuous extension to $H_{\rho,-1}(\R)\otimes H$, which will be denoted by $M(\partial_{0,\rho}^{-1})$ as well. More precisely, we may define
\begin{align*}
     M(\partial_{0,\rho}^{-1}) & : H_{\rho,-1}(\R)\otimes H \to H_{\rho,-1}(\R)\otimes H \\
                               & \phi \mapsto (H_{-\rho,1}(\R)\otimes H \ni \psi \mapsto \langle M(\partial_{0,\rho}^{-1})\partial_{0,\rho}^{-1}\phi, -\partial_{0,-\rho}\psi\rangle_{0,0}).
\end{align*}
This definition is justified by observing that $M(\partial_{0,\rho}^{-1})\partial_{0,\rho}^{-1} \phi =\partial_{0,\rho}^{-1} M(\partial_{0,\rho}^{-1}) \phi$ for all $\phi \in H_{\rho,0}(\R;H)$.
\end{enumerate}
\end{rem}
A prominent example of an analytic and bounded function of $\partial_{0,\rho}^{-1}$ is the delay operator, which itself is a special case of the time translation:
\begin{example}[Time translation]{\label{ex:time_translation}} Let $r\in \R_{>0}$, $\rho\in \R_{>1/2r}$, $h\in \R$ and $u\in H_{\rho,0}(\R)\otimes H$. We define
\[
   \tau_h u := u(\cdot + h).
\]
The operator $\tau_h \in L(H_{\rho,0}(\R)\otimes H,H_{\rho,0}(\R)\otimes H)$ is called \emph{time translation operator}. We note that $\tau_h \partial_{0,\rho}^{-1}=\partial_{0,\rho}^{-1}\tau_h$ holds. Therefore, there is a unique continuous extension of $\tau_h$ to the space $H_{\rho,-1}(\R)\otimes H$. We will use the same name for that extension. The operator norm of $\tau_h$ equals $\exp(h\rho)$. If $h<0$ the operator $\tau_h$ is also called a \emph{delay operator}. In that case the mapping
\[
  B_{\mathbb{C}}(r,r) \ni z \mapsto M(z):=\exp(z^{-1}h)1_H
\]
is analytic and uniformly bounded. An easy computation shows for $u \in H_{\rho,0}\otimes H$ that
\[
   \exp(\left(\partial_{0,\rho}^{-1}\right)^{-1}h)1_Hu=M(\partial_{0,\rho}^{-1})u = \mathbb{L}_\rho^* \exp((\mathrm{i}m+\rho)h) \mathbb{L}_\rho u = u(\cdot + h).
\]
\end{example}
Analytic, bounded, operator-valued functions of $\partial_{0,\rho}^{-1}$ have the following
properties, which can be found in \cite[Theorem 2.10]{Pi2009-1} or
\cite[Lemma 1.3.2]{Wau2011} (the link between the causality notions
is established in \cite[Lemma 1.2.3]{Wau2011}).
\begin{thm}\label{thm:funcalc}
Let $r\in\R_{>0}$, $H$ a Hilbert space and $M:B_{\mathbb{C}}(r,r)\to L(H)$
be an analytic, bounded mapping. Then for all $\rho\in \R_{>1/2r}$
the operator $M(\partial_{0,\rho}^{-1})$ is a causal (in the sense of \cite{Laksh10,PicMcG2011} also cp. Definition \ref{def:causality}), bounded linear operator in $H_{\rho,0}(\mathbb{R})\otimes H$.
\end{thm}
We shall discuss other possible analytic functions of $\partial_{0,\rho}^{-1}$ in Example \ref{exam:funcalc}. 

\section{Basic Solution Theory}\label{sec:BST}
 Let $H$ be a Hilbert space. In the framework prepared in the previous section we consider equations of the form
\begin{equation}
\partial_{0,\rho}u=F(u).\label{eq:ode}
\end{equation}
Before we desribe the properties of $F$, we introduce a particular type of test function space and a particular type of distributions.
\begin{defn}
  Denoting by $\supp \phi$ the support of a function $\phi$, we define
\[
    \interior C_{\infty}^+\left(\mathbb{R};H\right):= \{ \phi \in C_{\infty}\left(\mathbb{R};H\right);\sup \supp \phi <\infty \text{, there is }n\in\mathbb{N} \text{ with }\phi^{(n)} \in \interior C_{\infty}\left(\mathbb{R};H\right) \}
\]
and
\[
    \interior C_{\infty}^+\left(\mathbb{R};H\right)' := \{ \phi : \interior C_{\infty}^+\left(\mathbb{R};H\right) \to \mathbb{K}; \phi \text{ linear}\}.
\]
\end{defn}
The right-hand side $F$ in \eqref{eq:ode} is assumed to be a mapping
\[
   F : \interior C_{\infty}\left(\mathbb{R};H\right) \to \interior C_{\infty}^+\left(\mathbb{R};H\right)'.
\]
Moreover, assume the existence of $\rho_0\in \R_{>0}$ and $s\in (0,1)$ such that for all $\rho\in \R_{>\rho_0}$ there is $K\in\R_{>0}$ such that for all $u,w\in \interior C_{\infty}\left(\mathbb{R};H\right)$ and $\psi \in \interior C_{\infty}^+\left(\mathbb{R};H\right)$ we have
\begin{equation}\label{eq:estimate}
    \abs{F(0)(\psi)}\leq K \abs{\psi}_{-\rho,1}
    \quad \text{and} \quad 
    \abs{F(u)(\psi)-F(w)(\psi)} \leq s \abs{\psi}_{-\rho,1} \abs{u-w}_{\rho,0}.
\end{equation}
It is straightforward that for all $\rho\in \R_{>\rho_0}$ the mapping $F$ possesses a unique Lipschitz continuous extension $F_\rho$ from $H_{\rho,0}(\R)\otimes H$ to $H_{\rho,-1}(\R)\otimes H$. The respective Lipschitz constant may be chosen strictly less than $1$.
In order to obtain well-posedness of \eqref{eq:ode}, the task is in finding $\rho$ such that \eqref{eq:ode} admits a unique solution $u$, which continuously depends on the right hand side $F$ in some adapted sense.  In this setting (\ref{eq:ode}) should hold in
$H_{\rho,-1}(\mathbb{R})\otimes H$ noting that $\partial_{0,\rho}$ can be regarded as the unitary operator from $H_{\rho,0}(\mathbb{R})\otimes H$ onto
$H_{\rho,-1}(\mathbb{R})\otimes H$ (Theorem \ref{thm:sobolev chain}). In this situation we arrive at the following result.

\begin{thm}[Picard-Lindel\"of]{\label{thm:Picard-Lindeloef}} Let $\rho_0\in \R_{>0}$, $s\in (0,1)$ and let $F:\interior C_{\infty}\left(\mathbb{R};H\right) \to \interior C_{\infty}^+\left(\mathbb{R};H\right)'$ be such that the estimates \eqref{eq:estimate} hold for all $\rho\in \R_{>\rho_0}$. Then for all $\rho \in \R_{>\rho_0}$ there exists a uniquely determined $u\in
H_{\rho,0}(\mathbb{R})\otimes H$ such that
\begin{equation*}
\partial_{0,\rho} u= F_\rho(u) \mbox{ in }H_{\rho,-1}(\mathbb{R})\otimes H.
\end{equation*}
\end{thm}

\begin{proof} Let $\rho \in \R_{>\rho_0}$.
We consider the fixed point problem
\begin{equation}
u=\partial_{0,\rho}^{-1}F_\rho(u). {\label{eq:fixed point}}
\end{equation}
According to Theorem \ref{thm:sobolev chain} the operator $\partial_{0,\rho}^{-1}$ is unitary from
$H_{\rho,-1}(\mathbb{R})\otimes H$ to $H_{\rho,0}(\mathbb{R})\otimes H$. Since the Lipschitz constant of $F_\rho$ is strictly less than $1$, the contraction mapping theorem implies the existence of a unique $u\in H_{\rho,0}(\mathbb{R})\otimes
H$ satisfying (\ref{eq:fixed point}). For this fixed point it
follows that
\begin{equation*}
\partial_{0,\rho}u=F_\rho(u) \mbox{ in } H_{\rho,-1}(\mathbb{R})\otimes H.
\end{equation*}
The uniqueness of the solution follows immediately, since for any
element $x\in H_{\rho,0}(\mathbb{R})\otimes H$ satisfying
(\ref{eq:ode}) the fixed point equation (\ref{eq:fixed point})
holds as well. Since this fixed point is unique, we conclude $u=x$.
\end{proof}

If we discuss a particular class of equations of neutral type (see Theorem \ref{thm:neutral2}), it turns out that the respective function $F$ satisfies another estimate than \eqref{eq:estimate}. We consider other possible cases subsequently.

\begin{cor}{\label{cor:Picard-Lindeloef 2}}
Let $\rho_0 \in \R_{>0}$, $s\in (0,1)$ and let $F:\interior C_{\infty}\left(\mathbb{R};H\right) \to \interior C_{\infty}^+\left(\mathbb{R};H\right)'$ be such that for all $\rho\in \R_{>\rho_0}$, there is $K\in\R_{>0}$ such that for all $u,w\in \interior C_{\infty}\left(\mathbb{R};H\right)$ and $\psi \in \interior C_{\infty}^+\left(\mathbb{R};H\right)$ we have
\[
   \abs{F(0)(\psi)}\leq K\abs{\psi}_{-\rho,0}
   \quad \text{and} \quad 
   \abs{F(u)(\psi)-F(w)(\psi)} \leq s \abs{\psi}_{-\rho,0} \abs{u-w}_{\rho,1}.
\]
For $\rho \in \R_{>\rho_0}$ denote by $F_\rho: H_{\rho,1}(\mathbb{R})
\otimes H \to H_{\rho,0}(\mathbb{R})\otimes H$ the strictly contracting extension of $F$.
Then for all $\rho \in \R_{>\rho_0}$ there is a unique $u\in
H_{\rho,1}(\mathbb{R})\otimes H$ satisfying
\begin{equation*}
\partial_{0,\rho} u=F_\rho(u)
\end{equation*}
in $H_{\rho,0}(\mathbb{R})\otimes H$.
\end{cor}

\begin{proof}
Let $\rho \in \mathbb{R}_{>\rho_0}$. We consider the mapping
\begin{align*}
G: H_{\rho,0}(\mathbb{R})\otimes H & \to H_{\rho,-1}(\mathbb{R})\otimes H \\
                     v &\mapsto \partial_{0,\rho} F_\rho(\partial_{0,\rho}^{-1} v).
\end{align*}
Since the operators $\partial_{0,\rho}$ and $\partial_{0,\rho}^{-1}$ are
unitary, it follows that $G_\rho$ is strictly contracting. Thus, by Theorem
\ref{thm:Picard-Lindeloef} there exists a unique $v\in
H_{\rho,0}(\mathbb{R})\otimes H$ with
\begin{equation*}
\partial_{0,\rho} v=G(v)=\partial_{0,\rho} F_\rho(\partial_{0,\rho}^{-1}v)
\end{equation*}
which is equivalent to
\begin{equation*}
v=F_\rho(\partial_{0,\rho}^{-1}v).
\end{equation*}
By setting $u := \partial_{0,\rho}^{-1}v\in H_{\rho,1}(\mathbb{R})\otimes
H$, we obtain the desired solution of our differential equation.
The uniqueness is clear since any solution $x\in
H_{\rho,1}(\mathbb{R})\otimes H$ satisfies
\begin{equation*}
\partial_{0,\rho}x=F_\rho(x)=F_\rho(\partial_{0,\rho}^{-1}\partial_{0,\rho} x).
\end{equation*}
Hence, by the uniqueness of $v$ we obtain $\partial_{0,\rho}x=v$ and thus
$x=u$.
\end{proof}

If we impose a modified Lipschitz-type condition on $F$, it turns out that we gain
better regularity of the solution. It should be noted that in contrary to (\ref{eq:estimate}) this estimate does not impose a strict contractivity condition on $F$.

\begin{cor}{\label{cor:Picard-Lindeloef 3}}
Let $k\in \{0,1\}$, $\rho_0, C\in \R_{>0}$ and let $F: \interior C_{\infty}\left(\mathbb{R};H\right) \to \interior C_{\infty}^+\left(\mathbb{R};H\right)'$ be such that for all $\rho\in \R_{>\rho_0}$, there exists $K\in\R_{>0}$ such that for all $u,w\in \interior C_{\infty}\left(\mathbb{R};H\right)$ and $\psi \in \interior C_{\infty}^+\left(\mathbb{R};H\right)$ we have
\[
  \abs{F(0)(\psi)}\leq K\abs{\psi}_{-\rho,k}
  \quad\text{and}\quad
  \abs{F(u)(\psi) - F(w)(\psi)} \leq C \abs{\psi}_{-\rho,k}\abs{u-w}_{\rho,k}.
\]
For $\rho\in \R_{>\rho_0}$ we denote by $F_\rho: H_{\rho,k}(\mathbb{R}) \otimes H \to H_{\rho,k}(\mathbb{R})\otimes H$ the unique continuous extension of $F$.
Then for all $\rho\in \R_{>\max\{C,\rho_0\}}$ there is a unique $u\in H_{\rho,k+1}(\mathbb{R})\otimes H$ with
\begin{equation*}
\partial_{0,\rho} u=F_\rho(u) \mbox{ in } H_{\rho,k}(\mathbb{R})\otimes H.
\end{equation*}
\end{cor}

\begin{proof} Let $\rho \in \R_{>\max\{C,\rho_0\}}$. For $\psi \in \interior C_{\infty}^+\left(\mathbb{R};H\right)$, we observe that by Theorem \ref{thm:sobolev chain}
\[
 \abs{\psi}_{-\rho,k}=\abs{\partial_{0,-\rho}^{-1}\partial_{0,-\rho}\psi}_{-\rho,k}\leq \frac1{\rho}\abs{\partial_{0,-\rho}\psi}_{-\rho,k}=\frac1\rho\abs{\psi}_{-\rho,k+1}.
\]
Consequently, $F_\rho$ considered as a mapping with values in $H_{\rho,k-1}(\R)\otimes H$ is strictly contracting.
Thus, we are in the situation of Theorem \ref{thm:Picard-Lindeloef} if $k=0$ or
in the situation of Corollary \ref{cor:Picard-Lindeloef 2} in the
case $k=1$. Hence, we find a unique $u\in
H_{\rho,k}(\mathbb{R})\otimes H$ with
\begin{equation*}
\partial_{0,\rho}u=F_\rho(u).
\end{equation*}
Since $F_\rho(u)\in H_{\rho,k}(\mathbb{R})\otimes H$, it follows that
indeed $u\in H_{\rho,k+1}(\mathbb{R})\otimes H$ and the equation
even holds in $H_{\rho,k}(\mathbb{R})\otimes H$.
\end{proof}

\begin{rem}\label{rem: rho_Inde}
It should be noted that the solution of (\ref{eq:ode}) seems to
depend on the choice of the parameter $\rho\in \mathbb{R}_{>\rho_0}$.
This however is not the case and will be shown in Theorem \ref{thm: rho-Independent}.
\end{rem}
Now we show the continuous dependence of our solution $u$ of
(\ref{eq:ode}) on the function $F$ with respect to a suitable
topology. For a Lipschitz continuous mapping $F : H_{\rho,0}(\mathbb{R})\otimes H \to
H_{\rho,-1}(\mathbb{R})\otimes H$ we denote the best Lipschitz constant of $F$ by $|F|_{\mathrm{Lip}}$.
\begin{thm}{\label{thm:continuous dependence}}
Let $\rho \in \mathbb{R}_{>0}$ and
\begin{equation*}
F,G: H_{\rho,0}(\mathbb{R})\otimes H \to
H_{\rho,-1}(\mathbb{R})\otimes H
\end{equation*}
be two Lipschitz-continuous mappings with
\begin{equation*}
\frac{|F|_{\mathrm{Lip}}+|G|_{\mathrm{Lip}}}{2} < 1.
\end{equation*}
Furthermore let $u,v\in H_{\rho,0}(\mathbb{R})\otimes H$ with
\begin{equation*}
\partial_{0,\rho}u=F(u) \mbox{ and } \partial_{0,\rho} v=G(v).
\end{equation*}
Then
\begin{equation*}
|u-v|_{\rho,0} \leq
\frac{1}{1-\frac{|F|_{\mathrm{Lip}}+|G|_{\mathrm{Lip}}}{2}}
\sup_{x\in H_{\rho,0}(\mathbb{R})\otimes H}
|F(x)-G(x)|_{\rho,-1}.
\end{equation*}
\end{thm}

\begin{proof}
For $u$ and $v$ we compute in $H_{\rho,0}(\R)\otimes H$
\begin{align*}
u-v =& \partial_{0,\rho}^{-1} F(u)-\partial_{0,\rho}^{-1} G(v) \\
=& \frac{1}{2} \partial_{0,\rho}^{-1} (F(u)-F(v)) - \frac{1}{2}
\partial_{0,\rho}^{-1} (G(v)-F(v)) +\frac{1}{2} \partial_{0,\rho}^{-1}
F(u)-\frac{1}{2} \partial_{0,\rho}^{-1} G(v) \\
=& \frac{1}{2} \partial_{0,\rho}^{-1} (F(u)-F(v))
+\frac{1}{2}\partial_{0,\rho}^{-1}(G(u)-G(v)) -\frac{1}{2}\partial_{0,\rho}^{-1}
(G(v)-F(v)) +\frac{1}{2}\partial_{0,\rho}^{-1} (F(u)-G(u)).
\end{align*}
The latter yields
\begin{equation*}
|u-v|_{\rho,0} \leq
\frac{1}{2}(|F|_{\mathrm{Lip}}+|G|_{\mathrm{Lip}})|u-v|_{\rho,0} + \sup_{x\in H_{\rho,0}(\mathbb{R})\otimes H}
|F(x)-G(x)|_{\rho,-1}
\end{equation*}
and thus
\begin{equation*}
|u-v|_{\rho,0} \leq
\frac{1}{1-\frac{|F|_{\mathrm{Lip}}+|G|_{\mathrm{Lip}}}{2}}
\sup_{x\in H_{\rho,0}(\mathbb{R})\otimes H}
|F(x)-G(x)|_{\rho,-1}. \qedhere
\end{equation*}
\end{proof}

\begin{rem}\label{rem:contdep}~

 \begin{enumerate}[(a)]
  \item{\label{rem0}} In the case of
\begin{equation*}
 F,G:H_{\rho,1}(\mathbb{R})\otimes H \to H_{\rho,0}(\mathbb{R})\otimes H
\end{equation*}
with
\begin{equation*}
 \frac{|F|_{\mathrm{Lip}}+|G|_{\mathrm{Lip}}}{2} < 1
\end{equation*}
we consider the mappings $\tilde{F},\tilde{G}$ given by $\tilde{F}(v)=\partial_{0,\rho} F(\partial_{0,\rho}^{-1}v)$ and
$\tilde{G}(v)=\partial_{0,\rho} G(\partial_{0,\rho}^{-1}v)$ for $v\in H_{\rho,0}(\mathbb{R})\otimes H$. These mappings are again Lipschitz continuous from $H_{\rho,0}(\mathbb{R})\otimes H$ to $H_{\rho,-1}(\mathbb{R})\otimes H$ with
\begin{equation*}
 |\tilde{F}|_{\mathrm{Lip}}=|F|_{\mathrm{Lip}} 
 \quad\text{and}\quad
 |\tilde{G}|_{\mathrm{Lip}}=|G|_{\mathrm{Lip}}.
\end{equation*}
 Thus we have
\begin{equation*}
 |\tilde{u}-\tilde{v}|_{\rho,0} \leq
\frac{1}{1-\frac{|\tilde{F}|_{\mathrm{Lip}}+|\tilde{G}|_{\mathrm{Lip}}}{2}}
\sup_{x\in H_{\rho,0}(\mathbb{R})\otimes H}
|\tilde{F}(x)-\tilde{G}(x)|_{\rho,-1}
\end{equation*}
for $\partial_{0,\rho} \tilde{u}=\tilde{F}(\tilde{u})$ and $\partial_{0,\rho} \tilde{v}=\tilde{G}(\tilde{v})$ by Theorem \ref{thm:continuous dependence}. Since $u:=\partial_{0,\rho}^{-1}\tilde{u}$ and $v:=\partial_{0,\rho}^{-1}\tilde{v}$ satisfy $\partial_{0,\rho}u=F(u)$ and $\partial_{0,\rho} v=G(v)$ we conclude that, by using the unitarity of $\partial_{0,\rho}^{-1}$
\begin{equation*}
 |u-v|_{\rho,1} \leq
\frac{1}{1-\frac{|F|_{\mathrm{Lip}}+|G|_{\mathrm{Lip}}}{2}}
\sup_{x\in H_{\rho,1}(\mathbb{R})\otimes H}
|F(x)-G(x)|_{\rho,0}.
\end{equation*}
\item If
\begin{equation*}
 F,G:H_{\rho,k}(\mathbb{R})\otimes H \to H_{\rho,k}(\mathbb{R})\otimes H
\end{equation*}
with
\begin{equation*}
 \frac{|F|_{\mathrm{Lip}}+|G|_{\mathrm{Lip}}}{2} < \rho
\end{equation*}
for $k\in \{0,1\}$ we consider the functions
\begin{equation*}
 \tilde{F},\tilde{G} : H_{\rho,k}(\mathbb{R})\otimes H \to H_{\rho,k-1}(\mathbb{R})\otimes H
\end{equation*}
with $\tilde{F}(x)=F(x)$ and $\tilde{G}(x)=G(x)$ for $x\in H_{\rho,k}(\mathbb{R})\otimes H$. Then we can estimate the Lipschitz-constants of $\tilde{F}$ and $\tilde{G}$ by
$\rho^{-1}|F|_{\mathrm{Lip}}$ and $\rho^{-1}|G|_{\mathrm{Lip}}$, respectively. Thus we conclude
 \begin{equation*}
 \frac{|\tilde{F}|_{\mathrm{Lip}}+|\tilde{G}|_{\mathrm{Lip}}}{2} < 1
\end{equation*}
and hence, we can apply Theorem \ref{thm:continuous dependence} in the case $k=0$ and Remark \ref{rem:contdep}\eqref{rem0} in the case $k=1$. Thus for $u,v\in H_{\rho,k}(\mathbb{R})\otimes H$ with $\partial_{0,\rho}u=F(u)$ and $\partial_{0,\rho}v=G(v)$ we get
\begin{align*}
 |u-v|_{\rho,k} &\leq \frac{1}{1-\frac{|\tilde{F}|_{\mathrm{Lip}}+|\tilde{G}|_{\mathrm{Lip}}}{2}}
\sup_{x\in H_{\rho,k}(\mathbb{R})\otimes H}
|\tilde{F}(x)-\tilde{G}(x)|_{\rho,k-1}\\
&\leq \frac{1}{\rho-\frac{|F|_{\mathrm{Lip}}+|G|_{\mathrm{Lip}}}{2}}
\sup_{x\in H_{\rho,k}(\mathbb{R})\otimes H}
|F(x)-G(x)|_{\rho,k}.
\end{align*}
\end{enumerate}
\end{rem}

Another common case is that $F$ has two arguments, i.e.,
\begin{align*}
F: (H_{\rho,0}(\mathbb{R})\otimes H) \oplus (H_{\rho,-1}(\mathbb{R})\otimes H) &\to H_{\rho,-1}(\mathbb{R})\otimes H \\
(u,f) &\mapsto F(u,f),
\end{align*}
where the second argument $f$ should be interpreted as a certain source term, cf.\ Theorem \ref{thm:ivode}. Then we arrive at the problem of finding $u\in H_{\rho,0}(\mathbb{R})\otimes H$ such that
\begin{equation}{\label{eq:ode source term}}
 \partial_{0,\rho} u=F(u,f)
\end{equation}
for a fixed $f\in H_{\rho,-1}(\mathbb{R})\otimes H$. We now define what it means for (\ref{eq:ode source term}) to be autonomous.

\begin{defn}
An ordinary differential equation of the form (\ref{eq:ode source term})
 is called \emph{autonomous} if $F$ commutes with time translation,
i.e.,\ for all $u\in H_{\rho,0}(\mathbb{R})\otimes H,f\in H_{\rho,-1}(\mathbb{R})\otimes H$ and $h\in \mathbb{R}$  we have
\begin{equation*}
 F(\tau_h u, \tau_h f)=\tau_h F(u,f).
\end{equation*}
\end{defn}

\begin{rem}
In the situation of an autonomous equation (\ref{eq:ode source term}) it follows that for a solution $u\in H_{\rho,0}(\mathbb{R})\otimes H$ and for each $h\in \mathbb{R}$
\begin{equation*}
 \partial_{0,\rho}\tau_h u=\tau_h \partial_{0,\rho} u= \tau_h F(u,f)= F(\tau_h u,\tau_h f).
\end{equation*}
This means that the translated solution solves the equation for the translated source term $f$.
\end{rem}

\section{Causality and Memory }\label{Sec: CausMem}

\subsection{Causal Solution Operators}

If a solution of an evolutionary problem up to time $a$ only depends on the equation up to time $a$, then the solution operator is called causal or nonanticipative. Already Volterra implicitely used nonanticipative operators in his work on 	integral equations. Later Tychonoff made contributions in developing the theory of functional equations involving causal operators, cp.\ also \cite{Devi2011,Laksh10}. We consider equations of the form (\ref{eq:ode}) with time on the whole real line. In this setting causality is a natural property of the solution operator. At first we give a definition of causality in our framework.

\begin{defn}{\label{def:causality}}
Let $X,Y$ be Hilbert spaces, $\rho\in \R$. A mapping
\begin{equation*}
W:D\left(W\right)\subseteq H_{\rho,0}(\mathbb{R})\otimes X\to H_{\rho,0}(\mathbb{R})\otimes Y
\end{equation*}
is called \emph{causal} if for all $a\in\mathbb{R},\, x,y\in D\left(W\right)$\footnote{For a Hilbert space $H$ and a bounded, measurable function $\phi:\R\to \R$, we denote
\begin{equation*}
 (\phi(m_0)f)(t):=\phi(t)f(t)\quad (t\in \mathbb{R},f\in H_{\rho,0}(\mathbb{R})\otimes H).
\end{equation*}} 
\begin{equation*}
\left(\chi_{\R_{<a}}\left(m_{0}\right)\left(x-y\right)=0\implies\chi_{\R_{<a}}\left(m_{0}\right)\:\left(W\left(x\right)-W\left(y\right)\right)=0\right).
\end{equation*}
\end{defn}

\begin{rem}~

 \begin{enumerate}[(a)]
\item An equivalent formulation of causality is the following, cf.\ also \cite{Thomas1997,Weiss2003}.
A mapping \[
W:D\left(W\right)\subseteq H_{\rho,0}(\mathbb{R})\otimes X\to H_{\rho,0}(\mathbb{R})\otimes Y,\]
is causal, if for all $a\in \mathbb{R}$
\begin{equation*}
 \chi_{\R_{<a}}(m_0)W=\chi_{\R_{<a}}(m_0)W \chi_{\R_{<a}}(m_0).
\end{equation*}
\item If $\rho\neq 0$, then it is immediate from the formulas in Theorem \ref{thm:sobolev chain} that $\partial_{0,\rho}^{-1}$ is causal if and only if $\rho>0$.
\end{enumerate}
\end{rem}

It is remarkable that causality is actually implied by the uniform Lipschitz continuity we required in Theorem \ref{thm:Picard-Lindeloef}. In order to prove this fact, i.e., Theorem \ref{thm:causality}, we need a definition.
\begin{defn}
  Let $w \in \interior C_{\infty}^+\left(\mathbb{R};H\right)'$. Then define
 \[
     \int_{-\infty}^\cdot w : \interior C_{\infty}^+\left(\mathbb{R};H\right) \to \mathbb{R} : \psi \mapsto w(\int_{\cdot}^\infty \psi).
 \]
\end{defn}
\begin{rem} Let $\rho \in \R_{>0}$. We choose to identify $w \in \interior C_{\infty}^+\left(\mathbb{R};H\right)'$, for which there exists $C\in \R_{>0}$ such that for all $\psi\in \interior C_\infty^+(\R;H)$ it holds $\abs{w(\psi)}\leq C \abs{\psi}_{-\rho,1}$, with an element of $H_{\rho,-1}(\R)\otimes H$ by appropriate continuous extension. Then we have the equality
\[
   \int_{-\infty}^\cdot w  = \partial_{0,\rho}^{-1} w \in H_{-\rho,0}(\R)^*\otimes H \cong H_{\rho,0}(\R)\otimes H.
\]
Indeed, let $\psi \in \interior C_{\infty}^+\left(\mathbb{R};H\right)$. Then we have
\[
   \int_{-\infty}^\cdot w (\psi) = w(\int_{\cdot}^\infty \psi) = w(-\partial_{0,-\rho}^{-1}\psi) = \partial_{0,\rho}^{-1} w(\psi).
\]
\end{rem}

\begin{thm}{\label{thm:causality}}
With the assumptions and the notation from Theorem \ref{thm:Picard-Lindeloef}, we have that for all $\rho \in \mathbb{R}_{>\rho_0}$ the mapping $\partial_{0,\rho}^{-1}F_\rho$ is causal as a mapping from $H_{\varrho,0}(\R)\otimes H$ to $H_{\varrho,0}(\R)\otimes H$.
\end{thm}

\begin{proof} Let $a\in \R$, $\rho \in \R_{>\rho_0}$ and let $\phi \in C_\infty(\R)$ be bounded. Now, let $v \in  \interior C_{\infty}\left(\mathbb{R};H\right)$ and $\psi \in \interior C_{\infty}\left(\mathbb{R};H\right)$ be such that $\sup \supp \psi \leq a$. For $\eta \in \R_{\geq \rho}$ we compute
\begin{align*}
    \abs{\partial_{0,\rho}^{-1} F_\rho (v)(\psi) -  \partial_{0,\rho}^{-1}F_\rho(\phi(m_0) v)(\psi)} & = \abs{\int_{-\infty}^\cdot F(v)(\psi) - \int_{-\infty}^\cdot F(\phi(m_0) v)(\psi)} \\
    & = \abs{\partial_{0,\eta}^{-1} F_\eta (v)(\psi) -  \partial_{0,\eta}^{-1}F_\eta(\phi(m_0) v)(\psi)} \\
    & = \abs{F_\eta(v)(-\partial_{0,-\eta}^{-1}\psi)-F_\eta(\phi(m_0) v)(-\partial_{0,-\eta}^{-1} \psi)} \\
    & \leq \abs{-\partial_{0,-\eta}^{-1} \psi }_{-\eta,1}\abs{ v- \phi(m_0) v }_{\eta,0}.\\
    & = \abs{\psi }_{-\eta,0}\abs{ v- \phi(m_0) v }_{\eta,0} \\
    & \leq \abs{\psi}_{0,0} e^{\eta a}\abs{ v- \phi(m_0) v }_{\eta,0}.
\end{align*}
Summarizing, we get for $\eta \in \R_{\geq \rho}$
\[
   \abs{\partial_{0,\rho}^{-1} F_\rho (v)(\psi) -  \partial_{0,\rho}^{-1}F_\rho(\phi(m_0) v)(\psi)} \leq \abs{\psi}_{0,0} e^{\eta a}\abs{ v- \phi(m_0) v }_{\eta,0}.
\]
By continuity, we deduce for $\eta \in \R_{\geq \rho}$
\begin{align*}
     &\abs{\partial_{0,\rho}^{-1} F_\rho (v)(\psi) -  \partial_{0,\rho}^{-1}F_\rho(\chi_{\R_{<a}}(m_0) v)(\psi)}\\& \leq \abs{\psi}_{0,0} e^{\eta a}\abs{ v- \chi_{\R_{<a}}(m_0) v }_{\eta,0}  \\
      &  = \abs{\psi}_{0,0} e^{\eta a} \abs{\chi_{\R_{>a}}(m_0)v}_{\eta,0} = \abs{\psi}_{0,0} e^{\eta a} \left( \int_a^\infty \abs{v(t)}^2 e^{-2\eta t} d t\right)^{\frac12} \\
      &  = \abs{\psi}_{0,0} \left( \int_0^\infty \abs{v(t+a)}^2 e^{-2\eta t} d t\right)^{\frac12}.
\end{align*}
Letting $\eta \to \infty$ in the above inequality, we conclude that
\[
\abs{\partial_{0,\rho}^{-1} F_\rho (v)(\psi) -  \partial_{0,\rho}^{-1}F_\rho(\chi_{\R_{<a}}(m_0) v)(\psi)}=0.
                                                                    \]
Hence, by the choice of $\psi$, the function $\partial_{0,\rho}^{-1} F_\rho (v)-\partial_{0,\rho}^{-1}F_\rho(\chi_{(-\infty, a)}(m_0) v)$ is not supported on the set $\R_{<a}$. This yields the claim.
\end{proof}

Of course, as it was already noted in Remark \ref{rem: rho_Inde}, it is tempting to believe that the solution of \eqref{eq:ode} provided by the above Picard-Lindel\"of-type theorems (cf.\ Theorem \ref{thm:Picard-Lindeloef} and the Corollaries \ref{cor:Picard-Lindeloef 2} and \ref{cor:Picard-Lindeloef 3}) depends on the particular choice of $\rho$. This is not the case as our next result confirms.

\begin{thm}\label{thm: rho-Independent} With the assumptions and the notation from Theorem \ref{thm:Picard-Lindeloef}, the following holds. The respective solutions $w_{\rho_{k}}\in H_{\rho_{k},0}(\R)\otimes H$, $k\in \{1,2\}$, of \eqref{eq:ode} for $\rho_{1},\rho_{2}\geq\rho_{0}$, coincide, i.e., \[
w_{\rho_{1}}=w_{\rho_{2}}\in H_{\rho_{1},0}(\R)\otimes H\cap H_{\rho_{2},0}(\R)\otimes H\]
provided that
\[
   \partial_{0,\rho_1}w_{\rho_1}= F_{\rho_1}(w_{\rho_1}) 
   \quad\text{and}\quad 
   \partial_{0,\rho_2}w_{\rho_2}= F_{\rho_2}(w_{\rho_2})
\]
holds in $H_{\rho_{1},-1}(\R)\otimes H$ and $H_{\rho_{2},-1}(\R)\otimes H$, respectively.
 \end{thm}
 \begin{proof} Let $a\in\R$, $\rho\in \mathbb{R}_{\geq\rho_{0}}$. Denoting by $w_\rho$ the solution of
 \[
    \partial_{0,\rho}w_{\rho}= F_\rho(w_{\rho}) \in H_{\rho,-1}(\R)\otimes H,
 \]we recall $w_{\rho}\in H_{\rho,0}(\R)\otimes H$.
 Moreover, we have due to causality, i.e., Theorem \ref{thm:causality}
  \begin{eqnarray*}
\chi_{\R_{<a}}\left(m_{0}\right)w_{\rho} & = & \chi_{\R_{<a}}\left(m_{0}\right)\partial_{0,\rho}^{-1}F\left(w_{\rho}\right)\\
 & = & \chi_{\R_{<a}}\left(m_{0}\right)\partial_{0,\rho}^{-1}F\left(\chi_{\R_{<a}}\left(m_{0}\right)w_{\rho}\right).\end{eqnarray*}
 Let  $\rho\in \mathbb{R}_{\geq\rho_{0}}$ be such that $\min\{\rho_1,\rho_2\}\geq \rho$. Then, as $\partial_{0,\rho}^{-1}F_\rho$ leaves $H_{\rho,0}(\R)\otimes H$ invariant, we have for $k\in\{1,2\}$, since $\chi_{\R_{<a}}(m_0)w_{\rho_k} \in H_{\rho,0}(\R)\otimes H$, the following equality
 \[
\chi_{\R_{<a}}(m_0)\partial_{0,\rho_k}^{-1}F_{\rho_k}(\chi_{\R_{<a}}(m_0)w_{\rho_k})=\chi_{\R_{<a}}(m_0)\partial_{0,\rho}^{-1}F_\rho(\chi_{\R_{<a}}(m_0)w_{\rho_k}).
 \]
 Hence, \begin{eqnarray*}
 &  & \left\langle w_{\rho_{1}}-w_{\rho_{2}}\big|\chi_{\R_{<a}}\left(m_{0}\right)\left(w_{\rho_{1}}-w_{\rho_{2}}\right)\right\rangle_{\rho,0}=\\
 &  & \qquad=\left\langle w_{\rho_{1}}-w_{\rho_{2}}\big|\chi_{\R_{<a}}\left(m_{0}\right)\left(\partial_{0,\rho}^{-1}F_\rho\left(\chi_{\R_{<a}}\left(m_{0}\right)w_{\rho_{2}}\right)-\partial_{0,\rho}^{-1}F_\rho\left(\chi_{\R_{<a}}\left(m_{0}\right)w_{\rho_{2}}\right)\right)\right\rangle _{\rho,0}.\end{eqnarray*}
 Using the Cauchy-Schwarz inequality, we get
 \begin{align*}
& \left|\chi_{\R_{<a}}\left(m_{0}\right)\left(w_{\rho_{1}}-w_{\rho_{2}}\right)\right|_{\rho,0}^2 \\ &\leq \left|F\left(\chi_{\R_{<a}}\left(m_{0}\right)w_{\rho_{1}}\right)-F\left(\chi_{\R_{<a}}\left(m_{0}\right)w_{\rho_{2}}\right)\right|_{\rho,-1}\left|\chi_{\R_{<a}}\left(m_{0}\right)\left(w_{\rho_{1}}-w_{\rho_{2}}\right)\right|_{\rho,0}\\
 & \leq s \left|\chi_{\R_{<a}}\left(m_{0}\right)\left(w_{\rho_{1}}-w_{\rho_{2}}\right)\right|_{\rho,0}^2.\end{align*}
 Since $s<1$, we deduce $ \left|\chi_{\R_{<a}}\left(m_{0}\right)\left(w_{\rho_{1}}-w_{\rho_{2}}\right)\right|_{\rho,0}=0$.
 Since $a\in\mathbb{R}$ was arbitrary, the desired result follows.
\end{proof}

In the spirit of the continuous dependence result Theorem \ref{thm:continuous dependence}, we now show causality of the solution operator in a suitably adapted sense. Thereby, we strengthen the causality result by showing that the solution is independent of \emph{any} future of $F$. Before, however, stating the theorem, we define the set of possible right hand sides in \eqref{eq:ode} for some $\rho\in\R_{>0}$.


\begin{defn}\label{def:Lip} Let $H$ be a Hilbert space. For $\rho\in \R_{>0}$ we define $\textnormal{Con}_{\textnormal{ev}}(H_{\rho,0}(\R)\otimes H; H_{\rho,-1}(\R)\otimes H)$ the set of all eventually contracting mappings, i.e., $F \in \textnormal{Con}_{\textnormal{ev}}(H_{\rho,0}(\R)\otimes H; H_{\rho,-1}(\R)\otimes H)$ if and only if $F: H_{\rho,0}(\R)\otimes H \to H_{\rho,-1}(\R)\otimes H$ and there exists $s\in (0,1)$ such that for all $\eta \in \R_{\geq \rho}$, there is $K\in \R_{>0}$ such that for all $u,w\in \interior C_\infty(\R;H)$ and $\psi \in \interior C_\infty^+(\R;H)$ it holds
\[
    \abs{F(0)(\psi)}\leq K\abs{\psi}_{-\eta,1}
    \quad\text{and}\quad
    \abs{F(u)(\psi)-F(w)(\psi)} \leq s \abs{\psi}_{-\eta,1}\abs{u-w}_{\eta,0}.
\]
We may summarize the solution operators of \eqref{eq:ode} in the following way.
Define
\begin{align*}
   S_\rho : \,&\textnormal{Con}_{\textnormal{ev}}(H_{\rho,0}(\R)\otimes H; H_{\rho,-1}(\R)\otimes H) \to H_{\rho,0}(\R)\otimes H,
\end{align*}
where for $F\in \textnormal{Con}_{\textnormal{ev}}(H_{\rho,0}(\R)\otimes H; H_{\rho,-1}(\R)\otimes H)$ the element $S_\rho(F)\in H_{\rho,0}(\R)\otimes H$ is the unique fixed point of
\[
  x=\partial_{0,\rho}^{-1}F\left(x\right).
\]
Note that Theorem \ref{thm:Picard-Lindeloef} ensures the unique existence of this fixed point.
\end{defn}

We may formulate the strengthened causality result.

\begin{thm} Let $\rho\in \R_{>0}$, $a\in\R$ and let $H$ be a Hilbert space. Let $F,G \in \textnormal{Con}_{\textnormal{ev}}(H_{\rho,0}(\R)\otimes H; H_{\rho,-1}(\R)\otimes H)$. Assume $\chi_{\R_{<a}}\left(m_{0}\right)\partial_{0,\rho}^{-1}F=\chi_{\R_{<a}}\left(m_{0}\right)\partial_{0,\rho}^{-1}G$. Then  $\chi_{\R_{<a}}\left(m_{0}\right)S_\rho(F)=\chi_{\R_{<a}}\left(m_{0}\right)S_\rho(G)$.
\end{thm}
\begin{proof}
For $x,y \in H_{\rho,0}(\R)\otimes H$ we need to show $\chi_{\R_{<a}}\left(m_{0}\right)x=\chi_{\R_{<a}}\left(m_{0}\right)y$,
if \[
x=\partial_{0,\rho}^{-1}F\left(x\right),\: y=\partial_{0,\rho}^{-1}G\left(y\right).\]
Due to the causality of $\partial_{0,\rho}^{-1}F$ and $\partial_{0,\rho}^{-1}G$, cf. Theorem \ref{thm:causality}, we have\begin{eqnarray*}
\chi_{\R_{<a}}\left(m_{0}\right)x & = & \chi_{\R_{<a}}\left(m_{0}\right)\partial_{0,\rho}^{-1}F\left(\chi_{\R_{<a}}\left(m_{0}\right)x\right)\\
\chi_{\R_{<a}}\left(m_{0}\right)y & = & \chi_{\R_{<a}}\left(m_{0}\right)\partial_{0,\rho}^{-1}G\left(\chi_{\R_{<a}}\left(m_{0}\right)y\right)\\
 & = & \chi_{\R_{<a}}\left(m_{0}\right)\partial_{0,\rho}^{-1}F\left(\chi_{\R_{<a}}\left(m_{0}\right)y\right)\end{eqnarray*}
We see that $\chi_{\R_{<a}}\left(m_{0}\right)x$ and $\chi_{\R_{<a}}\left(m_{0}\right)y$
are both solutions of a fixed point problem for the same contractive
mapping\[
\chi_{\R_{<a}}\left(m_{0}\right)\partial_{0,\rho}^{-1}F,\]
which implies \[
\chi_{\R_{<a}}\left(m_{0}\right)x=\chi_{\R_{<a}}\left(m_{0}\right)y.\qedhere\]
\end{proof}

\subsection{Delay and Memory}\label{Subsec: DelayMem}

In this section we provide a new definition for operators having delay. In order to do so, we define the opposite, i.e., we introduce the concept of an operator being memoryless or amnesic. It is remarkable that this notion is dual to the concept of causality.
\begin{defn} Let $\rho\in\R$ and $X, Y$ be Hilbert spaces.
A mapping\[
W:D\left(W\right)\subseteq H_{\rho,0}(\R)\otimes X\to H_{\rho,0}(\R)\otimes Y,\]
 is called \emph{amnesic} or said to \emph{have no delay} if for all $a\in\mathbb{R}$ and $x,y\in D\left(W\right)$
\[
\left(\chi_{_{\mathbb{R}_{>a}}}\left(m_{0}\right)\left(x-y\right)=0\implies\chi_{_{\mathbb{R}_{>a}}}\left(m_{0}\right)\:\left(W\left(x\right)-W\left(y\right)\right)=0\right).
\]
If a mapping $W$ is not amnesic, we also say $W$ \emph{has memory}
or \emph{has delay}.
\end{defn}

We observe that, by the very definition, a first example for an amnesic operator is $\partial_{0,\rho}^{-1}$ for $\rho<0$ or $(\partial_{0,\rho}^{-1})^*$ for $\rho>0$. We may also give other examples of amnesic operators, namely operators of Nemitzki type.

\begin{example}[Nemitzki operators]\label{ex:Nem} Let $H$ be a Hilbert space, $f : \R \times H \to H$, $\rho \in \R_{>0}$. We assume that $f$ is uniformly Lipschitz continuous with respect to the first variable, i.e., there exists $L>0$ such that for all $t\in\R$ and $x,y\in H$ we have
\[
   \abs{f(t,x)-f(t,y)}_H \leq L \abs {x - y}_H.
\]
Moreover, assume $f(t,0)=0$ for all $t\in\R$.
We may define the following mapping
\[
   F_\rho : H_{\rho,0}(\R)\otimes H \to H_{\rho,0}(\R)\otimes H : u \mapsto (t\mapsto f(t,u(t))).
\]
The uniform Lipschitz continuity of $f$ together with $f(t,0)=0$ for all $t\in\R$ ensures that $F_\rho$ is well-defined. We claim that $F_\rho$ is amnesic. Indeed, let $a\in \R$ and $u,v \in  H_{\rho,0}(\R)\otimes H$ be such that $\chi_{\R_{>a}}(m_0)(u-v)=0$. Then for a.e. $t\in\R_{>a}$ we have
\[
   F_\rho(t,u(t))= F_\rho(t, \chi_{\R_{>a}}(t)u(t))=F_\rho(t,\chi_{\R_{>a}}(t)v(t))=F_\rho(t,v(t)).
\]
Thus, the claim follows.
\end{example}

An example of an operator, which has memory is $\partial_{0,\rho}^{-1}$ for $\rho>0$. Now, we want to \emph{define}, when a differential equation of the form \eqref{eq:ode} is a delay differential equation, which, to the best of our knowledge, has not been done yet in a mathematically rigorous way. For a definition of delay differential equations of the form \eqref{eq:ode}, one should take into account that this should only depend on the right hand side $F$. However, this right hand side is in general not of the form to be described as amnesic or to have delay, since $F$ maps in general into a space with negative index. Taking into account that the operator $(\partial_{0,\rho}^{-1})^*$ is amnesic for $\rho>0$ and the application of $(\partial_{0,\rho}^{-1})^*$ transforms $F$ into a mapping within $H_{\rho,0}(\R)\otimes H$ (cf. Remark \ref{rem: dualdescr}\eqref{rem: dualdescr2}), we arrive at the following possible definition.
\begin{defn} Let $H$ be a Hilbert space, $\rho>0$ and $F \in \textnormal{Con}_{\textnormal{ev}}(H_{\rho,0}(\R)\otimes H; H_{\rho,-1}(\R)\otimes H)$. A differential equation of the form (\ref{eq:ode}), i.e.,
\[
   \partial_{0,\rho} u = F(u)
\] is called a \emph{delay differential equation} if $\left(\partial_{0,\rho}^{-1}\right)^{*}F$ has delay.
\end{defn}
We illustrate this definition by means of the following examples.
\begin{example} Let $H$ be a Hilbert space and let $f: H\to H$ satisfy analogue conditions as in Example \ref{ex:Nem}, i.e., $f$ is Lipschitz continuous with Lipschitz constant $L>0$ and $f(0)=0$. Let $\rho\in \R_{>L}$, $g\in H_{\rho,-1}(\R)\otimes H$. We consider a differential equation of the following form\footnote{Such type of problems will be discussed later, when we come to initial value problems.}, with $F_\rho$ being analogously defined as in Example \ref{ex:Nem},
\begin{equation}\label{eq:amn}
   \partial_{0,\rho} u = F_\rho (u)+g.
\end{equation}
The latter equation admits a unique solution $u\in H_{\rho,0}(\R)\otimes H$ by our choice of $\rho$, cf. Corollary \ref{cor:Picard-Lindeloef 3}. Moreover, this differential equation is not a delay differential equation, since, as a composition of amnesic mappings, the mapping $u\mapsto (\partial_{0,\rho}^{-1})^*(F_\rho(u)+g)$ is amnesic itself. We may change equation \eqref{eq:amn} a little by introducing a time translation. Let $h\in\R_{>0}$ and $\tau_{-h}$ be the time translation of Example \ref{ex:time_translation}. Consider the differential equation
\begin{equation}\label{eq:del}
   \partial_{0,\rho} u = F_\rho (\tau_{-h} u) + g.
\end{equation}
The equation \eqref{eq:del} is a delay differential equation, if $f$ is not constant. Indeed, if $f$ is not constant, then there is $x_1,x_2 \in H$ such that $f(x_1)\neq f(x_2)$. Define $u := \chi_{[-h,0]}(\cdot)x_1, v := \chi_{[-h,0]}(\cdot)x_2$. Then $\chi_{\R_{>0}}(m_0)(u-v)=0$. Using the fact that $(\partial_{0,\rho}^{-1})^*$ is amnesic, we get
\begin{align*}
   & \chi_{\R_{>0}}(m_0)((\partial_{0,\rho}^{-1})^*(F_\rho(\tau_{-h}u)+g)-(\partial_{0,\rho}^{-1})^*(F_\rho(\tau_{-h} v)+g)) \\
   & = \chi_{\R_{>0}}(m_0)((\partial_{0,\rho}^{-1})^*(F_\rho(\tau_{-h} u))-(\partial_{0,\rho}^{-1})^*(F_\rho(\tau_{-h} v)))\\
   & = \chi_{\R_{>0}}(m_0)((\partial_{0,\rho}^{-1})^*\chi_{\R_{>0}}(m_0)(F_\rho(\tau_{-h} u))-(\partial_{0,\rho}^{-1})^*\chi_{\R_{>0}}(m_0)(F_\rho(\tau_{-h} v)))\\
   & = \chi_{\R_{>0}}(m_0) (\partial_{0,\rho}^{-1})^*\chi_{\R_{>0}}(m_0)(F_\rho(\tau_{-h} u)-F_\rho(\tau_{-h} v)).
\end{align*}
Since for a.e. $t\in [0, h]$ we have $(F_\rho(\tau_{-h} u)(t)-F_\rho(\tau_{-h} v)(t))=f(x_1)-f(x_2)\neq 0$, we deduce
\[
   \chi_{\R_{>0}}(m_0)(F_\rho(\tau_{-h} u)-F_\rho(\tau_{-h} v))\neq 0.
\]
Hence, as $(\partial_{0,\rho}^{-1})^*$ is one-to-one, we have
\[
  \chi_{\R_{>0}}(m_0) (\partial_{0,\rho}^{-1})^*\chi_{\R_{>0}}(m_0)(F_\rho(\tau_{-h} u)-F_\rho(\tau_{-h} v))\neq 0.
\]
This yields that \eqref{eq:del} is a delay differential equation. This justifies a posteriori the name delay operator for $\tau_{-h}$ since, by setting $f$ to be the identity on $H$, the equation \eqref{eq:del} is indeed a delay differential equation.
\end{example}

\section{Applications}\label{sec:app}

In this section we illustrate the versatility of the concepts developed above. In particular we give several examples of delay problems in order to show that many delay problems fit into the unified framework which we developed. We start with a discussion on how to investigate initial value problems within our context.

\subsection{Initial Value Problems for ODE}

We discuss the initial value problem
\begin{equation}\label{eq:ivode}
   \partial_0 u = F(u), \quad u(0) = u_0.
\end{equation}
This problem seems not to be covered by our previous reasoning, however, as we will show now, our abstract solution theory also applies to this initial value problem. Since our approach basically builds on $L^2$-space based arguments, we need a theorem, which justifies point-wise evaluation of functions. In order to do so, we define a weighted H\"older-type space.
\begin{defn} Let $\rho\in \R_{>0}$, $H$ Hilbert space. Then we define for a continuous function $\phi \in C(\R;H)$
\begin{align*}
   \left|\phi\right|_{\rho,\infty,1/2} & :=  \sup\left\{ \left|\exp\left(-\rho t\right)\phi\left(t\right)\right|_{H}\:\big|\, t\in\mathbb{R}\right\}\\  & \qquad + \sup\left\{ \left.\frac{\left|\exp\left(-\rho t\right)\phi\left(t\right)-\exp\left(-\rho s\right)\phi\left(s\right)\right|_{H}}{\left|t-s\right|^{1/2}}\:\right|\, t,s\in\mathbb{R}\:\wedge t\not=s\right\}.
\end{align*}
Moreover, define $C_{\rho,\infty,1/2}(\R;H):= \{ \phi \in C(\R;H); \abs{\phi}_{\rho,\infty,1/2}<\infty\}$. The vector space $C_{\rho,\infty,1/2}(\R;H)$ becomes a Banach space under the norm $\abs{\cdot}_{\rho,\infty,1/2}$.
\end{defn}

We have the following form of a Sobolev embedding result.
\begin{lem}[{\cite[Lemma 3.1.59]{PicMcG2011}}] \label{lem:ex:trace}Let $\rho\in\mathbb{R}\setminus\{0\}$. Then the mapping
\begin{eqnarray*}
  \interior C_{\infty}\left(\mathbb{R}; H\right)\subseteq H_{\rho,1}\left(\mathbb{R}\right)\otimes H & \to &C_{\rho,\infty,1/2}(\R;H) \\
u & \mapsto & \left(\R \ni t\mapsto u\left(t\right) \in H \right)\end{eqnarray*}
 has a continuous extension $\Gamma$ to all of $H_{\rho,1}\left(\mathbb{R}\right)\otimes H$
(sometimes called the {}``trace operator'' or the {}``operator
of point-wise evaluation in time'').
Moreover, for all $u\in H_{\rho,1}\left(\mathbb{R}\right)\otimes H$
\[
\sup\left\{ \left|\exp\left(-\rho t\right)\left(\Gamma
u\right)\left(t\right)\right|_{H}\:\big|\, t\in\mathbb{R}\right\}
\leq\frac{1}{\sqrt{2\abs{\rho}}}\:\left|u\right|_{H_{\rho,1}\left(\mathbb{R}\right)\otimes
H}\]
 and\[
\sup\left\{ \left.\frac{\left|\exp\left(-\rho t\right)\left(\Gamma u\right)\left(t\right)-\exp\left(-\rho s\right)\left(\Gamma u\right)\left(s\right)\right|_{H}}{\left|t-s\right|^{1/2}}\:\right|\, t,s\in\mathbb{R}\:\wedge t\not=s\right\} \leq\left|u\right|_{H_{\rho,1}\left(\mathbb{R}\right)\otimes H}.\]
 Furthermore, the mapping $\Gamma$ is injective.
\end{lem}
\begin{proof} Let $\phi\in\interior C_{\infty}(\mathbb{R};H)$.
We have for $s,t \in \R$ with $s\leq t$ invoking H\"older's inequality and the mean value theorem \begin{eqnarray*}
\left|\phi\left(t\right)-\phi\left(s\right)\right|_{H} & = & \left|\int_{s}^{t}\partial_{0,\rho}\phi\left(u\right)\: du\right|_{H}\\
 & \leq & \sqrt{\left|\int_{s}^{t}\:\exp\left(2\rho u\right)\: du\right|}\sqrt{\left|\int_{s}^{t}\left|\partial_{0,\rho}\phi\left(u\right)\right|_{H}^{2}\:\exp\left(-2\rho u\right)\: du\right|}\\
 & \leq & \sqrt{\frac{\left|\exp\left(2\rho t\right)-\exp\left(2\rho s\right)\right|}{2\abs{\rho}}}\left|\phi\right|_{\rho,1}\\
 & \leq & \sqrt{\left|t-s\right|}\:\sqrt{\frac{\left|\exp\left(2\rho t\right)-\exp\left(2\rho s\right)\right|}{2\abs{\rho}\left|t-s\right|}}\;\left|\phi\right|_{\rho,1}\\
 & \leq & \sqrt{\left|t-s\right|}\:\max\left\{ \exp\left(\rho x\right)\big|\: x\in\{s,t\} \right\} \;\left|\phi\right|_{\rho,1}\end{eqnarray*}
 from which we can read off the desired H\"older continuity. With $s\to-\infty$
we also see from the second inequality that\[
\left|\phi\left(t\right)\right|\leq\exp\left(\rho
t\right)\sqrt{\frac{1}{2\abs{\rho}}}\left|\phi\right|_{\rho,1}.\]
 Moreover, using the relations $\partial_{0,\rho} =\partial_\rho + \rho$, cf.\ Corollary \ref{cor:FLT}, and $\abs{\phi}_{\rho,1}^2 = \abs{\rho \phi}_{\rho,0}^2+ \abs{\partial_\rho \phi }_{\rho,0}^2$, we calculate  \begin{eqnarray*}
\left|\exp\left(-\rho t\right)\phi\left(t\right)-\exp\left(-\rho s\right)\phi\left(s\right)\right|_{H} & = & \left|\int_{s}^{t}\left(\partial_{0,\rho}\left(\exp\left(-\rho m_{0}\right)\phi\right)\right)\left(u\right)\: du\right|_{H}\\
 & \leq & \sqrt{\left|t-s\right|}\sqrt{\left|\int_{s}^{t}\left|\left(\partial_{0,\rho}\left(\exp\left(-\rho m_{0}\right)\phi\right)\right)\left(u\right)\right|_{H}^{2}\: du\right|}\\
 & \leq & \sqrt{\left|t-s\right|}\sqrt{\left|\int_{s}^{t}\left|\left(\partial_{0,\rho}\phi-\rho\phi\right)\left(u\right)\right|_{H}^{2}\:\exp\left(-2\rho u\right)\: du\right|}\\
 & \leq & \sqrt{\left|t-s\right|}\sqrt{\left|\int_{s}^{t}\left|\left(\left(\partial_{0,\rho}-\rho\right)\phi\right)\left(u\right)\right|_{H}^{2}\:\exp\left(-2\rho u\right)\: du\right|}\\
 & \leq & \sqrt{\left|t-s\right|}\:\left|\left(\partial_\rho\right)\phi\right|_{\rho,0},\\
 & \leq & \sqrt{\left|t-s\right|}\:\sqrt{\left|\left(\partial_\rho\right)\phi\right|_{\rho,0}^{2}+\left|\rho\phi\right|_{\rho,0}^{2}}\\
 & = & \sqrt{\left|t-s\right|}\:\left|\phi\right|_{\rho,1},\end{eqnarray*}
 which shows that the mapping under consideration is a well-defined
continuous linear mapping. This mapping can now be extended by the
obvious uniform continuity to all of $H_{\rho,1}\left(\mathbb{R}\right)\otimes H$
due to the density of $\interior C_{\infty}\left(\mathbb{R};H\right)$
in $H_{\rho,1}\left(\mathbb{R}\right)\otimes H$.

Finally, to see that $\Gamma:\: H_{\rho,1}\left(\mathbb{R}\right)\otimes H\to C_{\rho,\infty,1/2}(\R;H) $
is injective, assume $\left(\phi_{k}\right)_{k}$ is a sequence in
$\interior C_{\infty}\left(\mathbb{R};H\right)$
with the property that $\left|\phi_{k}\right|_{\rho,\infty,1/2}\overset{k\to\infty}{\to}0$
is a Cauchy sequence in $H_{\rho,1}\left(\mathbb{R}\right)\otimes H$
with limit $f\in H_{\rho,1}\left(\mathbb{R}\right)\otimes H$. We
need to show that $f=0$. Letting $k\to\infty$ in the equality\begin{eqnarray*}
\left\langle \partial_{0,\rho}\phi_{k}\big|\,\psi\right\rangle_{0,0} & = & \left\langle \phi_{k}\big|\,-\partial_{0,-\rho}\psi\right\rangle _{0,0}\:,\end{eqnarray*}
 where $\psi\in\interior C_{\infty}\left(\mathbb{R};H\right)$
is arbitrary, we obtain\begin{eqnarray*}
\left\langle \partial_{0,\rho}f\big|\,\psi\right\rangle _{0,0} & = & 0\end{eqnarray*}
 for every $\psi\in\interior C_{\infty}\left(\mathbb{R};H\right)$
from which we conclude that \begin{eqnarray*}
\partial_{0,\rho}f & = & 0\end{eqnarray*}
 and hence\begin{eqnarray*}
f & = & 0\textrm{ in }H_{\rho,1}\left(\mathbb{R}\right)\otimes H\end{eqnarray*}
 follows. \end{proof}

Lemma \ref{lem:ex:trace} gives a criterion when it may be reasonable
to impose initial conditions. The solution of the respective
differential equation has to lie in some sense in the space
$H_{\rho,1}(\R)\otimes H$. We need the following definition.
\begin{defn}[Dirac delta distribution] Let $\rho\in\R_{>0}$. Then
$\chi_{\R_{>0}} \in H_{\rho,0}(\R)$. We define the Dirac delta
distribution $\delta$ in the point zero as the derivative of the
Heavyside function:
\[
\delta :=   \partial_{0,\rho} \chi_{\R_{>0}}.
\]
Clearly, $\delta \in H_{\rho,-1}(\R)$. Moreover, it is easy to see
that
\[
   \delta : H_{-\rho,1}(\R)\to \mathbb{K} : \phi \mapsto \phi(0).
\] For a Hilbert space $H$ and $w\in H$ we denote by $\delta\otimes
w \in H_{\rho,-1}(\R)\otimes H$ the derivative of $t\mapsto
\chi_{\R_{>0}}(t)w$.
\end{defn}

Now our perspective on initial value problems is as follows.

\begin{thm}\label{thm:ivode} Let $\rho_0 \in \R_{>0}$, $H$ a Hilbert space, $C>0$. Let $F: \interior
C_\infty(\R;H) \to \interior C_\infty^+(\R;H)'$ be such that for all $\rho\in\R_{>\rho_0}$ there exists $K\in \R_{>0}$ such that for all
$u,w\in \interior C_\infty(\R;H)$ and $\psi \in \interior
C_\infty^+(\R;H)$ the estimates
\[
  \abs{F(0)(\psi)}\leq K \abs{\psi}_{-\rho,0}
  \quad\text{and}\quad 
  \abs{F(u)(\psi)-F(w)(\psi)} \leq C
   \abs{\psi}_{-\rho,0}\abs{u-w}_{\rho,0}
\]
hold. Moreover, assume that $F(\phi)=0$ for all $\phi\in \interior
C_\infty(\R;H)$ with $\supp \phi\subseteq (-\infty,0)$. Let $\rho \in \R_{>\rho_0}$, $u_0\in H$ and denote by $F_\rho:
H_{\rho,0}(\R)\otimes H \to H_{\rho,0}(\R)\otimes H$ the unique
Lipschitz continuous extension of $F$. Then the
equation
\[
   \partial_{0,\rho} u = F_{\rho}(u) + \delta\otimes u_0
\]
admits a unique solution $u\in H_{\rho,0}(\R)\otimes H$ such that $u-\chi_{\R_{>0}}(m_0)u \in
H_{\rho,1}(\R)\otimes H$ and $u(0+)=u_0$.
\end{thm}
\begin{proof}  The unique existence of $u \in H_{\rho,0}(\R)\otimes H$ follows from Theorem \ref{thm:Picard-Lindeloef}.
Moreover, from the equation that is satisfied by $u$ we get
\[
u =    \partial_{0,\rho}^{-1} (F_{\rho}(u) + \delta\otimes u_0).
\]
Therefore,
\[
   u-\chi_{\R_{>0}}u = u - \partial_{0,\rho}^{-1}\delta\otimes u_0 =
   \partial_{0,\rho}^{-1} F_\rho(u).
\]
Hence, we read off that
\[
 t\mapsto u(t)-\chi_{\mathbb{R}_{\ge0}}(t)u_{0}
 \]
 lies in $H_{\rho,1}\left(\mathbb{R}\right)\otimes H$ and thus in $C_{\rho,\infty,1/2}(\R;H)$, by Lemma \ref{lem:ex:trace}. As a consequence,
\[
 \begin{array}{ccl}
u(0+)-u_{0} & = & \left(u-\chi_{\mathbb{R}_{\ge0}}\otimes u_{0}\right)(0+)\\
 & = & \left(u-\chi_{\mathbb{R}_{\ge0}}\otimes u_{0}\right)(0-)\\
 & = & u(0-)\end{array}.\]
Theorem \ref{thm:causality} yields the causality of
$\partial_{0,\rho}^{-1}F_\rho$. It follows that $u\left(0-\right)=0$ and
therefore the initial condition\[ u\left(0+\right)=u_{0}\]  is satisfied.
\end{proof}

The fact that the solution depends continuously on the initial data is formulated in our setting as follows.

\begin{thm} Let $H$ be a Hilbert space, $C,D\in\R_{>0}$. Let $F, G : \interior C_\infty(\R;H)\to \interior
C_\infty^+(\R;H)'$ be such that for all $\rho\in \R_{>\max\{C,D\}}$, there exists $K\in\R_{>0}$ such that for all
$u,w\in \interior C_\infty(\R;H)$ and $\psi \in \interior
C_\infty^+(\R;H)$ we have
\[
   \abs{F(u)(\psi)-F(w)(\psi)}\leq C
   \abs{\psi}_{-\rho,0}\abs{u-w}_{\rho,0} \text{ and }   \abs{G(u)(\psi)-G(w)(\psi)}\leq
   D \abs{\psi}_{-\rho,0}\abs{u-w}_{\rho,0}.
\]
and
\[
   \abs{F(0)(\psi)}\leq K \abs{\psi}_{-\rho,0}
   \quad\text{and}\quad 
   \abs{G(0)(\psi)}\leq K \abs{\psi}_{-\rho,0}
\]Let $u_0, w_0 \in H$ and denote by $F_\rho, G_\rho$ the respective extensions of $F$ and $G$ as continuous mappings within
$H_{\rho,0}(\R)\otimes H$. Moreover, let $u,w \in
H_{\rho,0}(\R)\otimes H$ be the respective solutions of the
differential equations \[
  \partial_{0,\rho}u=F_\rho(u)+\delta\otimes u_0
   \quad\text{and}\quad
  \partial_{0,\rho}w=G_\rho(w)+\delta \otimes w_0.
\] Then the continuous dependence estimate
\[
\left|u-w\right|_{\rho,0}\le
\frac{1}{\left(2\varrho-\left(C+D\right)\right)}\left(
\sqrt{2\rho}|u_0-w_0|_{H}+2\sup_{x\in H_{\rho,0}(\R)\otimes
H}\abs{F_\rho(x)-G_\rho(x)}_{\rho,0}\right)\] holds.
\end{thm}
\begin{proof} Define $\tilde F_\rho : H_{\rho,0}(\R)\otimes H \to H_{\rho,-1}(\R)\otimes H:  u\mapsto F_\rho(u)+\delta\otimes u_0$ and analogously $\tilde G_\rho$. Applying Theorem \ref{thm:continuous dependence} and using Corollary \ref{cor:FLT}, we get
\begin{align*}
  \abs{u-w}_{\rho,0} & \leq \frac{1}{1-\frac{\abs{\tilde F_\rho}_{\textnormal{Lip}}+\abs{\tilde G_\rho}_{\textnormal{Lip}}}{2}} \sup_{x\in H_{\rho,0}(\R)\otimes H}\abs{\tilde F_\rho(x)-\tilde G_{\rho}(x)}_{\rho,-1} \\
                     & \leq \frac{1}{1-\frac{\abs{F_\rho}_{\textnormal{Lip}}+\abs{G_\rho}_{\textnormal{Lip}}}{2\rho}} \sup_{x\in H_{\rho,0}(\R)\otimes H}\abs{\partial_{0,\rho}^{-1}\left(F_\rho(x)+\delta\otimes u_0-G_\rho(x)+\delta\otimes u_0\right)}_{\rho,0} \\
                     & \leq \frac{2\rho}{2\rho-(C+D)}\left(\abs{\chi_{R_{>0}}\otimes (u_0-w_0)}_{\rho,0} + \sup_{x\in H_{\rho,0}(\R)\otimes H}\abs{\partial_{0,\rho}^{-1}\left(F_\rho(x)-G_\rho(x)\right)}_{\rho,0}\right) \\
                     & \leq \frac{2\rho}{2\rho-(C+D)}\left(\frac{1}{\sqrt{2\rho}}\abs{u_0-w_0}_{H} + \frac{1}{\rho}\sup_{x\in H_{\rho,0}(\R)\otimes H}\abs{F_\rho(x)-G_\rho(x)}_{\rho,0}\right).\qedhere
\end{align*}
\end{proof}

\subsection{Local Solvability}

It appears that the above solution theory only deals with global
solutions. This is, however, not the case. To illustrate how to get local existence results we consider an
initial value problem in a Hilbert space $H$
 \begin{align}
\partial_{0}u(t) & =g(t,u(t)),\label{eq:loc_ode}\\
u(0) & =u_{0},\nonumber \end{align}
 where $g:[0,T]\times H\to H$ is a measurable function. Moreover, for every $x\in H$ it holds
$g(\cdot,x)\in L^{\infty}([0,T],H)$ and there exists a radius $\eta\in\mathbb{R}_{>0}$ and a constant
$L\in\mathbb{R}_{>0}$ such that for all $y,z\in\overline{B_{H}(x,\eta)}=\{ w\in H; \abs{w-x}_H \leq \eta\}$ we have \begin{equation}
|g(\cdot,y)-g(\cdot,z)|_{L^{\infty}([0,T];H)}\leq L|y-z|_{H}.\label{eq:Lipschitz_cont}
\end{equation}
 In this situation we derive the following Lemma.
 \begin{lem} Let $u_{0}\in H$ and $\eta\in\mathbb{R}_{>0}$
such that (\ref{eq:Lipschitz_cont}) is satisfied for each $y,z\in\overline{B_{H}(u_{0},\eta)}$.
We denote the projection on the closed, convex set $\overline{B_{H}(u_{0},\eta)}$
by $P$. Then for each $\rho\in\mathbb{R}_{>0}$ the operator $F_{\rho}$
defined by \begin{align*}
F_{\rho}:C([0,T];H)\cap\left(H_{\rho,0}(\mathbb{R})\otimes H\right)\subseteq H_{\rho,0}(\mathbb{R})\otimes H & \to H_{\rho,0}(\mathbb{R})\otimes H\\
u & \mapsto(t\mapsto\chi_{[0,T]}(t)g(t,P(u(t))))\end{align*}
 is a Lipschitz-continuous mapping with $\limsup_{\rho\to\infty}|F_{\rho}|_{\mathrm{Lip}}<\infty$. Moreover, the continuous extension of $F_\rho$, for which we will use the same name is causal.
 \end{lem}
\begin{proof} For $u,v\in C([0,T];H)\cap H_{\rho,0}(\mathbb{R})\otimes H$
we estimate \begin{align*}
\int_{\mathbb{R}}|F_{\rho}(u)(t)-F_{\rho}(v)(t)|_{H}^{2}e^{-2\rho t}dt & =\int_{0}^{T}|g(t,P(u(t)))-g(t,P(v(t)))|_{H}^{2}e^{-2\rho t}dt\\
 & \leq L^{2}\int_{0}^{T}|P(u(t))-P(v(t))|_{H}^{2}e^{-2\rho t}dt\\
 & \leq L^{2}\int_{0}^{T}|u(t)-v(t)|_{H}^{2}e^{-2\rho t}dt\\
 & \leq L^{2}|u-v|_{\rho,0}^{2}.\end{align*}
 This would prove the Lipschitz continuity, if we ensure that $F_{\rho}$
is well-defined. This, however, follows by using the above estimate
to obtain \begin{align*}
|F_{\rho}(u)|_{\rho,0} & \leq|F_{\rho}(u)-F_{\rho}(\chi_{\mathbb{R}_{\geq0}}\otimes u_{0})|_{\rho,0}+|F_{\rho}(\chi_{\mathbb{R}_{\geq0}}\otimes u_{0})|_{\rho,0}\\
 & \leq L^{2}|u-\chi_{\mathbb{R}_{\geq0}}\otimes u_{0}|_{\rho,0}+|g(\cdot,u_{0})|_{L^{\infty}([0,T];H)}\sqrt{\frac{1-e^{-2\rho T}}{2\rho}}<\infty.\end{align*}
 The causality of $F_\rho$ is straightforward.
 \end{proof}

This Lemma shows that $F_{\rho}$ satisfies the conditions of our
solution theorem concerning initial value problems Theorem \ref{thm:ivode}. Thus, there is $\rho_0\in\R_{>0}$ such that for all $\rho\in\R_{>\rho_0}$ we find a unique solution $v\in H_{\rho,0}(\mathbb{R})\otimes H$
of \begin{equation}
\partial_{0,\rho}v=F_{\rho}(v)+\delta\otimes u_{0}.\label{eq:global_eq}\end{equation}
The next theorem asserts that a solution to \eqref{eq:global_eq} satisfies Equation \eqref{eq:loc_ode} at least for some non-vanishing time interval.
\begin{thm} Let $\rho\in \R_{>\rho_{0}}$ and let $v\in H_{\rho,0}(\mathbb{R})\otimes H$ be
the solution of (\ref{eq:global_eq}). Then there exists $t_{\ast}\in]0,T]$
such that $v$ satisfies Equation (\ref{eq:loc_ode}) on the interval
$[0,t_{\ast}]$. \end{thm}

\begin{proof} For each $t\in[0,T]$ we obtain due to causality of
$F_{\rho}$ \begin{align*}
\partial_{0,\rho}(\chi_{\mathbb{R}_{\leq t}}(m_{0})v+\chi_{\mathbb{R}_{>t}}\otimes v(t)) & =\chi_{\mathbb{R}_{\leq t}}(m_{0})\partial_{0,\rho}v\\
 & =\chi_{\mathbb{R}_{\leq t}}(m_{0})F_\rho(v)+\delta\otimes u_{0}\\
 & =\chi_{\mathbb{R}_{\leq t}}(m_{0})F_\rho(\chi_{\mathbb{R}_{\leq t}}(m_{0})v)+\delta\otimes u_{0}\end{align*}
 and thus we estimate using the first inequality in Lemma \ref{lem:ex:trace} \begin{align*}
 & \sqrt{2\rho}\sup\{|v(s)-u_{0}|_{H}e^{-\rho s}\,|\, s\in[0,t]\}\\
 & \leq|\partial_{0,\rho}(\chi_{\mathbb{R}_{\leq t}}(m_{0})v+\chi_{\mathbb{R}_{>t}}\otimes v(t)-\chi_{\mathbb{R}_{\geq0}}\otimes u_{0})|_{\rho,0}\\
 & =|\chi_{\mathbb{R}_{\leq t}}(m_{0})F_\rho(\chi_{\mathbb{R}_{\leq t}}(m_{0})v)|_{\rho,0}\\
 & \leq|\chi_{\mathbb{R}_{\leq t}}(m_{0})F_\rho(\chi_{\mathbb{R}_{\leq t}}(m_{0})v)-\chi_{\mathbb{R}_{\leq t}}(m_{0})F_\rho(\chi_{[0,t]}\otimes u_{0})|_{\rho,0}+|\chi_{\mathbb{R}_{\leq t}}(m_{0})F_\rho(\chi_{[0,t]}\otimes u_{0})|_{\rho,0}\\
 & \leq L|\chi_{\mathbb{R}_{\leq t}}(m_{0})v-\chi_{[0,t]}\otimes u_{0}|_{\rho,0}+|g(\cdot,u_{0})|_{L^{\infty}([0,t];H)}\sqrt{\frac{1-e^{-2\rho t}}{2\rho}}\\
 & \leq L\sup\{|v(s)-u_{0}|_{H}e^{-\rho s}\,|\, s\in[0,t]\}\sqrt{t}+|g(\cdot,u_{0})|_{L^{\infty}([0,t];H)}\sqrt{\frac{1-e^{-2\rho t}}{2\rho}}.\end{align*}
 If we choose $t<\frac{2\rho}{L^2}$ we can conclude that \[
\sup\{|v(s)-u_{0}|_{H}e^{-\rho s}\,|\, s\in[0,t]\}\leq\frac{\sqrt{\frac{1-e^{-2\rho t}}{2\rho}}}{\sqrt{2\rho}-L\sqrt{t}}|g(\cdot,u_{0})|_{L^{\infty}([0,t];H)}.\]
 Hence, \[
\sup\{|v(s)-u_{0}|_{H}\,|\, s\in[0,t]\}\leq e^{\rho t}\frac{\sqrt{\frac{1-e^{-2\rho t}}{2\rho}}}{\sqrt{2\rho}-L\sqrt{t}}|g(\cdot,u_{0})|_{L^{\infty}([0,t];H)}.\]
 Therefore, we can find $t_{\ast}\in[0,T]$ such that \[
\sup\{|v(s)-u_{0}|_{H}\,|\, s\in[0,t_{\ast}]\}\leq\eta,\]
 or in other words $v[[0,t_{\ast}]]\subseteq\overline{B_{H}(u_{0},\eta)}$
which implies \[
\partial_{0,\rho}v(t)=F_\rho(v)(t)=g(t,v(t))\]
 for each $t\in(0,t_{\ast})$. Since the initial condition $v(0)=u_{0}$
is satisfied by definition of $v$ and Theorem \ref{thm:ivode}, the assertion follows. \end{proof}

\subsection{Classical Delay Equations}

In this section we apply the existence theory to classical delay equations, in particular integro-differential and neutral equations. As a first step we observe that delay equations typically show a special structure which can be utilized to formulate natural assumptions for existence of solutions.

\subsubsection{A Structural Observation}\label{sec:strob}

In many applications it turns out, that the function $F$ of Theorem
\ref{thm:Picard-Lindeloef} factorizes as $F=\Phi\circ\Theta$, where
for Hilbert spaces $H$ and $V$ we have
\[ \Phi: \bigcap_{\eta \in \R_{>0}} H_{\eta,0}(\R)\otimes V \to \interior C_\infty^+(\R;H)'
\]
 and
\[
\Theta : \interior C_\infty(\R;H) \to \bigcap_{\eta \in \R_{>0}} H_{\eta,0}(\R)\otimes V
\]
 with appropriate Hilbert spaces $H$ and $V$ and suitable conditions on $\Phi$ and $\Theta$, which we specify later on.
Thus, we arrive at a specialized form of the general problem (\ref{eq:ode})
given by \[
\partial_{0}u=\Phi\left(\Theta u\right).\]
 We will prove well-posedness results for two particular cases. The
first one, Theorem \ref{thm:discrete}, describes discrete delay, the second one, Theorem \ref{thm:full_past}, deals with the whole past of $u$.
\begin{thm}\label{thm:discrete}
Let $N\in\mathbb{N}$, $H$ Hilbert space, let
$\theta_{0},\ldots\theta_{N-1}\in \mathbb{R}_{\leq0}$ be distinct,
$s\in (0,1)$, $\rho_{0}\in \R_{>0}$, $\Phi : \interior
C_\infty(\R;H^N)\to \interior C_\infty^+(\R;H)'$. Assume that for
all $\rho \in \R_{>\rho_0}$, there is $K\in\R_{>0}$ such that for all $u,w\in \interior C_\infty(\R;H^N)$ and
$\psi \in \interior C_\infty^+(\R;H)$ we have
\[
   \abs{\Phi(0)(\psi)}\leq K\abs{\psi}_{-\rho,1}
   \quad\text{and}\quad
   \abs{\Phi(u)(\psi)-\Phi(w)(\psi)}\leq s \abs{\psi}_{-\rho,1}
   \abs{u-w}_{\rho,0}.
\]
Denote by $\Phi_{\rho}$ the continuous extension of $\Phi$ as a
mapping from $H_{\rho,0}\left(\mathbb{R}\right)\otimes H^{N}$ to $
H_{\rho,-1}\left(\mathbb{R}\right)\otimes H$. For $\rho\in\R_{>\rho_{0}}$
let $\Theta_{\rho}:H_{\rho,0}\left(\mathbb{R}\right)\otimes H\to
H_{\rho,0}\left(\mathbb{R}\right)\otimes H^{N}$ be given by \[
\Theta_{\rho}x=\left(\tau_{\theta_{0}}x,\ldots,\tau_{\theta_{N-1}}x\right)\in
H_{\rho,0}\left(\mathbb{R}\right)\otimes H^{N}.\]
 Then, for $\rho$ large enough, the equation \[
\partial_{0,\rho}u=\Phi_{\rho}\left(\Theta_{\rho}(u)\right)\]
 admits a unique solution $u\in H_{\rho,0}(\mathbb{R})\otimes H$. Moreover, the solution operator is causal.\end{thm}
\begin{proof} Let $\rho\in \R_{>\rho_0}$. We observe that $\lVert\tau_{h}\rVert_{H_{\rho,0}(\R)\otimes H \to
H_{\rho,0}(\R) \otimes H} \leq e^{\rho h}$ for all $h\leq 0$, cf.\ Example
\ref{ex:time_translation}. Thus, $\limsup_{\rho\to\infty}\abs{\Theta_\rho}_{\text{Lip}}\leq 1$.
Hence, $F:= \Phi_\rho \circ \Theta_\rho$ satisfies the assumptions
of Theorem \ref{thm:Picard-Lindeloef} if $\rho$ is chosen large enough. Causality follows from
Theorem \ref{thm:causality}.
\end{proof}
We introduce the mapping, which assigns to a function the respective past.
\begin{defn} Let $H$ be a Hilbert space. For a function $\varphi : \R \to H$, we define
\begin{eqnarray*}
\varphi_{\left(\cdot\right)}:\; H^{\mathbb{R}} & \to & \left(H^{\mathbb{R}_{<0}}\right)^{\mathbb{R}}\\
\varphi & \mapsto & \left(t\mapsto\left(\theta\mapsto\varphi\left(t+\theta\right)\right)\right).\end{eqnarray*}
\end{defn}

Without additional effort, we can easily consider classical delay equations involving the whole past of $u$, e.g.,\ equations with unbounded delay.

\begin{thm}\label{thm:full_past} Let $H$ be a Hilbert space, $C,\rho_0\in \R_{>0}$, $s\in(0,1/2)$. Let $\Phi : \bigcap_{\eta\in\R_{>0}} H_{\eta,0}(\R)\otimes L_2(\R_{<0};H) \to \interior C_\infty^+(\R;H)'$ be such that for $\rho\in \R_{>\rho_0}$, there is $K\in\R_{>0}$ such that for all $u,w\in \bigcap_{\eta\in\R_{>0}} H_{\eta,0}(\R)\otimes L_2(\R_{<0};H)$ and $\psi \in \interior C_\infty^+(\R;H)$ we have
\[
   \abs{\Phi(0)(\psi)}\leq K\abs{\psi}_{-\rho,1}
   \quad\text{and}\quad
   \abs{\Phi(u)(\psi) - \Phi(w)(\psi)}\leq C \rho^s \abs{\psi}_{-\rho,1} \abs{u-w}_{\rho,0}.
\]
For $\rho \in \R_{>\rho_0}$ let $\Phi_\rho$ denote the Lipschitz continuous extension of $\Phi$ as a mapping from $H_{\rho,0}(\R)\otimes H$ to $H_{\rho,-1}(\R)\otimes H$. Then for $\rho \in \R$ such that $\rho>\max\{ \left(\frac C{\sqrt 2}\right)^{\frac{2}{1-2s}},\rho_0\}$ the equation
\[
   \partial_{0,\rho} u = \Phi_\rho(u_{(\cdot)})
\]
admits a unique solution. Moreover, the solution operator is causal.
\end{thm}
\begin{proof} Introduce the mapping
\[
  \Theta : \interior  C_\infty^+(\R;H) \to \bigcap_{\eta\in \R_{>0}} H_{\eta,0}(\R)\otimes L^2(\R_{<0};H): \varphi\mapsto \varphi_{(\cdot)}.
\]
We compute a possible Lipschitz constant for $\Theta$ considered as a mapping from $H_{\rho,0}(\R)\otimes H$ to $H_{\rho,0}(\R)\otimes L^2(\R_{<0};H)$ for $\rho\in \R_{>0}$. Let $u,w \in \interior C_\infty^+(\R;H)$. Then we have
\begin{align*}
 \abs{\Theta(u)-\Theta(w)}_{\rho,0}^{2}&= \int_{\mathbb{R}}\int_{\mathbb{R}_{<0}}\abs{u(t+\theta)-w(t+\theta)}_{H}^{2}d\theta \exp(-2\rho t)dt\\
 & = \int_{\mathbb{R}_{<0}}\int_{\mathbb{R}}\abs{u(t+\theta)-w(t+\theta)}_{H}^{2}\exp(-2\rho(t+\theta))dt\exp(2\rho\theta)d\theta\\
 & = \frac{1}{2\rho}\abs{u-w}_{\rho,0}^{2}.\end{align*}
 Now, let $\rho \in \R_{>\rho_0}$, $u,w \in \interior C_\infty^+(\R;H)$ and $\psi \in \interior C_\infty^+(\R;H)$. Then
 \[
    \abs{\Phi (\Theta u)(\psi)-\Phi(\Theta w)(\psi)}\leq C \rho^s \abs{\psi}_{-\rho,1}\abs{\Theta u-\Theta w}_{\rho,0}\leq \frac C{\sqrt 2}\rho^{s-1/2}\abs{\psi}_{-\rho,1}\abs{ u-w}_{\rho,0}.
 \]
 For $\rho > \max\{\left(\frac C{\sqrt 2}\right)^{\frac{2}{1-2s}},\rho_0\}$ the assumptions on $F$ in Theorem \ref{thm:Picard-Lindeloef} are satisfied for the choice $F= \Phi\circ \Theta$. Causality follows from Theorem \ref{thm:causality}.
\end{proof}

Note that the Lipschitz continuity assumption on $\Phi$ in Theorem \ref{thm:full_past} allows the Lipschitz constant to grow moderately for increasing $\rho$.

\begin{example}\label{ex:full_past}~

\begin{enumerate}[(a)]
  \item\label{ex:full_past1} The assumptions of Theorem \ref{thm:full_past} are satisfied for the following situation. Let $H$ be a Hilbert space, $g:\mathbb{R}\times L_{2}(\mathbb{R}_{<0};H)\to H$
with the property that there is $L,K\in\R_{>0}$ such that for all $t\in\mathbb{R},x,y\in L_{2}(\mathbb{R}_{<0})\otimes H$ we have
\[
   \abs{g(t,0)}\leq K
   \quad\text{and}\quad
   \lvert g(t,x)-g(t,y)\rvert_H\leq L\lvert x-y\rvert.
\]
Then $\Phi$ given by $\Phi(u) := (t\mapsto g(t,u(t)))$ satisfies the assumption of Theorem \ref{thm:full_past}.
  \item In turn, a possible choice for $g$ may be as follows. Let $h:\mathbb{R}\times\mathbb{R}_{<0}\times H\to H$ be such that for all $t\in\R$, $\theta \in \R_{<0}$ and $x,z\in H$ we have \[
\left|h\left(t,\theta,x\right)-h\left(t,\theta,z\right)\right|_{H}\leq L\left|x-z\right|_{H}\]
and $h(t,\theta,0)=0$.
 Then $g:\mathbb{R}\times L_{2}(\mathbb{R}_{<0})\otimes H\to H$ defined
as \[
g(t,x)=\int_{-\infty}^{0}h(t,\theta,x(\theta))d\theta\]
satisfies the condition in Example \ref{ex:full_past}\eqref{ex:full_past1}.
\end{enumerate}
\end{example}

\begin{rem}[Initial value problems for delay differential equations] Let $C, \rho_0\in \R_{>0}$ and let $\Phi : \bigcap_{\eta\in \R_{>0}} H_{\eta,0}(\R)\otimes L^2(\R_{<0};H) \to \interior C_\infty^+(\R;H)$ satisfy the following. For all $\rho\in \R_{>\rho_0}$ there exists $K\in\R_{>0}$ such that for all $u,w\in \interior C_\infty(\R;L^2(\R_{<0};H))$ and $\psi\in \interior C_\infty^+(\R;H)$ it holds
\[
   \abs{\Phi(0)(\psi)}\leq K\abs{\psi}_{-\rho,0}
   \quad\text{and}\quad
   \abs{\Phi(u)(\psi)-\Phi(w)(\psi)}\leq C\abs{\psi}_{-\rho,0}\abs{u-w}_{\rho,0}.
\]
Denoting the Lipschitz continuous extension of $\Phi$ to a mapping from $H_{\rho,0}(\R)\otimes L^2(\R_{<0};H)$ to $H_{\rho,0}(\R)\otimes H$ by $\Phi_\rho$, we want to discuss a delay differential equation for a given $x_{-\infty} \in L^2(\R_{<0};H)\cap C(\R_{<0};H)$
\begin{equation}\label{rem:eq:ivdde1}
   \partial_{0,\rho} u(t) = \Phi (u_{(t)}) \textnormal{ for }t\in \R_{>0} 
   \quad\text{and}\quad 
   u_{(0)}= x_{-\infty}.
\end{equation}
Consider the following problem
\begin{equation}\label{rem:eq:ivdde2}
   \partial_{0,\rho} w = \chi_{\R_{>0}}(m_0)\Phi((x_{-\infty})_{(\cdot)}+w_{(\cdot)}) +\delta \otimes x_{-\infty}(0-).
\end{equation}
Then, according to Theorem \ref{thm:ivode}, there exists a unique solution $w \in H_{\rho,0}(\R)\otimes H$ of Equation \eqref{rem:eq:ivdde2} such that $w-\chi_{\R_{>0}}(m_0)w \in H_{\rho,1}(\R)\otimes H$ holds. Moreover, causality of $\partial_{0,\rho}^{-1}$ implies that $\supp w \subseteq \R_{\geq 0}$. Indeed, we have
\begin{align*}
   \chi_{\R_{<0}}(m_0)w & = \chi_{\R_{<0}}(m_0)\partial_{0,\rho}^{-1} \left(\chi_{\R_{>0}}(m_0)\Phi((x_{-\infty})_{(\cdot)}+w_{(\cdot)}) +\delta \otimes x_{-\infty}(0-)\right) \\
                        & = \chi_{\R_{<0}}(m_0)\partial_{0,\rho}^{-1} \chi_{\R_{>0}}(m_0)\Phi((x_{-\infty})_{(\cdot)}+w_{(\cdot)}) +\chi_{\R_{<0}}(m_0)\partial_{0,\rho}^{-1} \delta \otimes x_{-\infty}(0-) \\
                        & = \chi_{\R_{<0}}(m_0)\partial_{0,\rho}^{-1} \chi_{\R_{<0}}(m_0)\chi_{\R_{>0}}(m_0)\Phi((x_{-\infty})_{(\cdot)} + w_{(\cdot)}) \\
                        & \quad +\chi_{\R_{<0}}(m_0)\chi_{\R_{>0}}(m_0) \otimes x_{-\infty}(0-)\\
                        & = 0.
\end{align*}
  By construction, $w$ satisfies the initial condition $w(0+) = x_{-\infty}(0-)$. Moreover, by setting
\[
   u : t\mapsto \begin{cases}  x_{-\infty}(t), & t\in\R_{<0}\\
   														 w(t), & t\in \R_{\geq 0}
                \end{cases}
\]
we have $\partial_{0,\rho} w(t) = \Phi (u_{(t)}) \textnormal{ for }t\in \R_{>0}$. Thus, $u$ is the desired solution of \eqref{rem:eq:ivdde1}. The uniqueness of $u$ follows from the uniqueness of $w$.
\end{rem}

\subsubsection{Integro-Differential Equations}


In this section we want to give examples for differential equations
leading to causal solution operators. The latter is closely related
to functions of the time derivative or its inverse in the sense of Definition \ref{def:funcalc}.
We state some examples of functions of the time-derivative; we refer
to \cite{Pi2009-1} and \cite{Wau2011}.
\begin{example}\label{exam:funcalc}
~
\begin{enumerate}[(a)]
\item Convolutions (1): Consider $k\in L^{2}(\mathbb{R})\cap L^{1}(\mathbb{R})$
with $\inf\supp k=0$ as a convolution kernel. Then, by the Paley-Wiener
theorem (cf. \cite[Chapter 19]{Rudin}), the Fourier transform $\hat{k}$ of $k$ belongs to the Hardy-Lebesgue
space, in particular, $\hat{k}$ is analytic as a function of $\mathbb{R}-i\mathbb{R}_{>0}$ (the lower complex half plane)
to $\mathbb{C}$. Moreover, since $k\in L_{1}(\mathbb{R})$, the operator
$k*\colon\varphi\mapsto k*\varphi$ is a bounded operator in $H_{\rho,0}(\mathbb{R})$
for any $\rho\in \R_{>0}$. In order to prove that $k*$ is a function of $\partial_{0,\rho}^{-1}$,
it suffices to consider the following. For $g\in H_{\rho,0}(\mathbb{R})$,
we have \[
k*g=\mathcal{L}_{\rho}^{*}\mathcal{L}_{\rho}(k*g)=\sqrt{2\pi}\mathcal{L}_{\rho}^{*}\left(\mathcal{L}_{\rho}k\mathcal{L}_{\rho}g\right)=\sqrt{2\pi}\mathcal{L}_{\rho}^{*}\hat{k}(m_0-i\rho)\mathcal{L}_{\rho}g,\]
 thus interpreting $\mathcal{L}_{\rho}k$ as the multiplication operator $\hat{k}(m_0-i\rho)$,
we are in the situation of $M$ above, where $M(z)=\hat{k}(-i\frac{1}{z})$.
\item Convolutions (2): Let $\varepsilon\in \R_{>0}$, $M:B_{\mathbb{C}}(0,\varepsilon)\to L(H)$
analytic. In \cite[Theorem 1.5.6 and Remark 1.5.7]{Wau2011} it is
shown that for $\rho\in \R_{>\frac{2}{\varepsilon}}$ there is $k\in L^{1}(\mathbb{R},\exp(-\rho t)\mathrm{d}t;L(H))$
such that \[
\left(M(\partial_{0,\rho}^{-1})g\right)(t)=M(0)g(t)+\int_{\mathbb{R}}k(t-s)g(s) ds\quad(g\in\interior C_{\infty}\left(\mathbb{R},H\right)).\]

\item\label{exam:funcalc3} Time shift: For $r,h\in \R_{>0}$ consider \[
M(z):=\exp(-z^{-1}h)\quad(z\in B_{\mathbb{C}}(r,r)).\]
 Then for $\rho\in \R_{>\frac{1}{2r}}$ the operator $M(\partial_{0,\rho}^{-1})$
is given by $\tau_{-h}$, cf. Example \ref{ex:time_translation}.
\item Fractional integrals $\partial_{0,\rho}^{-\alpha}$, $\alpha\in\left[0,1\right],\rho\in \R_{>0}$.
\end{enumerate}
\end{example}
We summarize our findings in the following theorem.
\begin{thm}\label{thm:IntegroDiff}
Let $H$ be a Hilbert space, $r\in\mathbb{R}_{>0}$ and let $M,N:B_{\mathbb{C}}(r,r)\to L(H)$
be analytic and bounded. Furthermore, let $F: \interior C_\infty(\R;H) \to \interior C_\infty^+(\R;H)$ be as in the Picard-Lindel\"of Theorem \ref{thm:Picard-Lindeloef}. Moreover, assume that for $\rho \in \R_{>\rho_0}$ the respective continuous extensions $F_\rho$ of $F$ as mappings from $H_{\rho,0}(\R)\otimes H$ to $H_{\rho,-1}(\R)\otimes H$ satisfy
\[
   \limsup_{\rho\to\infty}\abs{F_\rho}_{\textnormal{Lip}}< \sup_{z\in B(r,r)}\lVert M(z)\rVert_{L(H)}.\sup_{z\in B(r,r)}\lVert N(z)\rVert_{L(H)}
\] Then there is $\rho_1\in \R_{>\max\{\rho_0,\frac1{2r}\}}$ such that for all $\rho\in\R_{>\rho_1}$ the equation \[
\partial_{0,\rho}u=M(\partial_{0,\rho}^{-1})F_\rho(N(\partial_{0,\rho}^{-1})u)\]
 admits a unique solution $u\in H_{\rho,0}(\mathbb{R})\otimes H$.
Moreover, the solution operator is causal. \end{thm}
\begin{proof} By the Remarks \ref{Rem: rho-Independent}\eqref{Rem: rho-Independent0} and \eqref{Rem: rho-Independent2} we have
\[
 \sup_{\rho\in \R_{>1/2r}}\Abs{M(\partial_{0,\rho}^{-1})}_{L(H_{\rho,-1}(\R)\otimes H;H_{\rho,-1}(\R)\otimes H)}\leq  \sup_{z\in B(r,r)}\lVert M(z)\rVert_{L(H)}
\]
and a similar estimate with $M$ replaced by $N$. Thus, for $\rho_1$ large enough $M(\partial_{0,\rho}^{-1})F_\rho(N(\partial_{0,\rho}^{-1})(\cdot)) : H_{\rho,0}(\R)\otimes H \to H_{\rho,-1}(\R)\otimes H$ is a contraction for all $\rho\in \R_{>\rho_1}$. Thus, Theorem \ref{thm:Picard-Lindeloef} applies. Causality follows from Theorem \ref{thm:causality} together with Remark \ref{Rem: rho-Independent}\eqref{Rem: rho-Independent1}.
\end{proof}

\subsubsection{Neutral Differential Equations}

In this section, we consider neutral differential equations, i.e.,\ equations in which the derivative of the solution is
evaluated at a point in the past. We emphasize that a solution theory
of examples of such equations is indeed covered by the results of
the previous section. A qualitative behavior of neutral differential equations of the following type has been discussed in \cite{Wau2011b,Wau2011a}. A solution theory may be stated as follows.
\begin{thm}\label{thm:neutral}
Let $H$ be a Hilbert space, $A,B,C\in L(H)$, $h_1,h_2,\rho_{0}\in \R_{>0}$, \\
 $f\in\bigcap_{\rho\geq\rho_{0}}H_{\rho,0}(\mathbb{R})\otimes H$.
Then the equation \begin{equation}
\partial_{0,\rho}u-C\partial_{0,\rho}\tau_{-h_1}u=Au+B\tau_{-h_2}u+f\label{Neutral}\end{equation}
 admits a unique solution $u\in H_{\rho,0}(\mathbb{R})\otimes H$
for $\rho$ large enough. Moreover, the solution operator is causal. \end{thm}
\begin{proof}
By the above theorem, it suffices to show that Equation \eqref{Neutral}
can be written in the form \[
\partial_{0,\rho}u=M(\partial_{0,\rho}^{-1})F(N(\partial_{0,\rho}^{-1})u)\]
 for suitable $F,M,N$. In order to construct $F,M,N$ we observe
for $\rho, h\in\R_{>0}$ the following \[
\tau_{-h}=\exp(-h\partial_{0,\rho})=\exp(-h\partial_{\rho})\exp(-h\rho).\]
 As $\partial_{\rho}$ is skew-selfadjoint, $\Abs{\exp(-h\partial_{\rho})}=1$ and
thus, $\Abs{\tau_{-h}}=\exp(-h\rho)$. Moreover, by the above Example \ref{exam:funcalc}\eqref{exam:funcalc3},
$\tau_{-h}$ is a bounded and analytic function of $\partial_{0,\rho}^{-1}$.
Choose $\rho\in \R_{>\rho_{0}}$ such that for $0<h\leq h_1,h_2$ it holds \[
\Abs{C\exp(-h\partial_{0,\rho})}\leq\Abs{C}\Abs{\exp(-h\partial_{0,\rho})}=\Abs{C}\exp(-h\rho)<1.\]
 Then $(1-C\tau_{-h})$ is boundedly invertible. As a composition,
the latter is a function of $\partial_{0,\rho}^{-1}$. We may rewrite \eqref{Neutral} as
\[
\partial_{0,\rho}(1-C\tau_{-h_1})u=(A+B\tau_{-h_2})u+f.\]
 By the choice of $\rho$ we get \[
\partial_{0,\rho}u=(1-C\tau_{-h_1})^{-1}((A+B\tau_{-h_2})u+f).\]
 For $M(\partial_{0,\rho}^{-1}):=(1-C\exp(-h_1\partial_{0,\rho}))^{-1}$, $N(\partial_{0,\rho}^{-1}):=(A+B\exp(-h_2\partial_{0,\rho}))$
and $F(v):=v+f$ for $v,f\in H_{\rho,0}(\mathbb{R})\otimes H$. Then Theorem \ref{thm:IntegroDiff} applies.
\end{proof}
As a concluding example, which illustrates the versatility and utility
of the concepts developed here, we consider a general class of neutral
differential equations. We note that our observation concerning the factorization occurs in slightly different
version than in Section \ref{sec:strob}.
\begin{thm}\label{thm:neutral2} Let $C,\rho_0\in \R_{>0}$, $H$ Hilbert space and let $\Phi: \bigcap_{\eta\in \R_{>0}} H_{\eta,0}\otimes(H\oplus H)\to \bigcap_{\eta\in \R_{>0}}H_{\eta,0}\otimes H$ be such that for all $\rho\in \R_{>\rho_0}$ there is $K\in\R_{>0}$ such that for all $u,w\in\bigcap_{\eta\in\R_{>0}}H_{\eta,0}(\R)\otimes (H\oplus H)$ we have
\[
   \abs{\Phi(0)}_{\rho,0}\leq K
   \quad\text{and}\quad
   \abs{\Phi(u)-\Phi(w)}_{\rho,0}\leq C\abs{u-w}_{\rho,0}.
\]
Let $\alpha:\mathbb{R}\to\mathbb{R}$, $\beta:\mathbb{R}\to\mathbb{R}$ be
bijective, Lipschitz continuous with bounded measurable derivatives
a.e.\ and $\alpha\left(s\right)\geq s$, $\beta\left(s\right)\geq s+\epsilon_{0}$
for some $\epsilon_{0}\in\mathbb{R}_{>0}$ and all $s\in\mathbb{R}$.
Denote by $\Phi_\rho$ the continuous extension of $\Phi$ as a mapping from $H_{\rho,0}(\R)\otimes (H\oplus H)$ to $H_{\rho,0}(\R)\otimes H$ for $\rho\in\R_{>\rho_0}$. Then there is $\rho_1\in \R_{>\rho_0}$ such that for $\rho\in\R_{\geq \rho_1}$, the equation \begin{equation}
\partial_{0,\rho}u=F\left(u\right):=\Phi_\rho\left(u\circ\alpha^{-1},\left(\partial_{0,\rho}u\right)\circ\beta^{-1}\right)\label{eq:neutral}\end{equation}
 admits a unique solution $u\in H_{\rho,0}\otimes H$. Furthermore
the solution operator is causal. \end{thm}
\begin{proof} Let $\rho\in \R_{>\rho_0}$, $u,w\in \interior C_\infty(\R;H)$. We confirm the condition on $F$ in Corollary \ref{cor:Picard-Lindeloef 2} for $k=1$.
We estimate \begin{eqnarray*}
\left|u\circ\alpha^{-1}-w\circ\alpha^{-1}\right|_{\rho,0}^{2} & = & \int_{\mathbb{R}}\left|u\left(\alpha^{-1}\left(t\right)\right)-w\left(\alpha^{-1}\left(t\right)\right)\right|_{H}^{2}\exp\left(-2\rho t\right)\: dt\\
 & = & \int_{\mathbb{R}}\left|u\left(s\right)-w\left(s\right)\right|_{H}^{2}\exp\left(-2\rho\alpha\left(s\right)\right)\:\left|\alpha^{\prime}\right|\left(s\right)\: ds\\
 & \leq & \left|\alpha^{\prime}\right|_{L^\infty(\R)}\:\int_{\mathbb{R}}\left|u\left(s\right)-w\left(s\right)\right|_{H}^{2}\exp\left(-2\rho s\right)\: ds\\
 & \leq & \left|\alpha^{\prime}\right|_{L^\infty(\R)}\:\left|u-w\right|_{\rho,0}^{2}\end{eqnarray*}
 and similarly \begin{eqnarray*}
\left|u\circ\beta^{-1}-w\circ\beta^{-1}\right|_{\rho,0}^{2} & = & \int_{\mathbb{R}}\left|u\left(s\right)-w\left(s\right)\right|_{H}^{2}\exp\left(-2\rho\beta\left(s\right)\right)\:\left|\beta^{\prime}\right|\left(s\right)\: ds\\
 & \leq & \left|\beta^{\prime}\right|_{L^\infty(\R)}\:\int_{\mathbb{R}}\left|u\left(s\right)-w\left(s\right)\right|_{H}^{2}\exp\left(-2\rho\left(s+\epsilon_{0}\right)\right)\: ds\\
 & \leq & \left|\beta^{\prime}\right|_{L^\infty(\R)}\:\exp\left(-2\rho\epsilon_{0}\right)\left|u-w\right|_{\rho,0}^{2}.\end{eqnarray*}
  Consequently, we get\begin{eqnarray*}
 &  & \left|\Phi\left(u\circ\alpha^{-1},\left(\partial_{0,\rho}u\right)\circ\beta^{-1}\right)-\Phi\left(w\circ\alpha^{-1},\left(\partial_{0,\rho}w\right)\circ\beta^{-1}\right)\right|_{\rho,0}\\
 &  & \leq C\left(\left|u\circ\alpha^{-1}-w\circ\alpha^{-1}\right|_{\rho,0}+\left|\left(\partial_{0,\rho}u\right)\circ\beta^{-1}-\left(\partial_{0,\rho}w\right)\circ\beta^{-1}\right|_{\rho,0}\right)\\
 &  & \leq C \sqrt{\left|\alpha^{\prime}\right|_{L^\infty(\R)}}\:\left|u-w\right|_{\rho,0}+C\sqrt{\left|\beta^{\prime}\right|_{L^\infty(\R)}}\:\exp\left(-\rho\epsilon_{0}\right)\left|\partial_{0,\rho}u-\partial_{0,\rho}w\right|_{\rho,0}\\
 &  & \leq C \frac{1}{\rho}\,\sqrt{\left|\alpha^{\prime}\right|_{L^\infty(\R)}}\:\left|u-w\right|_{\rho,1}+C\sqrt{\left|\beta^{\prime}\right|_{L^\infty(\R)}}\:\exp\left(-\rho\epsilon_{0}\right)\left|u-w\right|_{\rho,1}\end{eqnarray*}
 Hence, $F : H_{\rho,1}(\R)\otimes H\to H_{\rho,0}(\R)\otimes H$ satisfies the condition in Corollary \ref{cor:Picard-Lindeloef 2}.
 Causality follows from Theorem \ref{thm:causality}.
\end{proof}
\begin{example} A typical instance of our present problem class is
$H=\mathbb{R}^{N}$, $N\in\mathbb{N},$ and $\Phi\left(u,v\right)=(t\mapsto g\left(t,u,v\right))$
with $g$ satisfying a uniform Lipschitz condition on $\mathbb{R}\times\mathbb{R}^{N}\times\mathbb{R}^{N}$\[
\left|g\left(t,u,v\right)-g\left(t,x,y\right)\right|\leq L\left(\left|u-x\right|+\left|v-y\right|\right)\qquad ((t,u,v),(t,x,y)\in \R\times \R^N\times \R^N).\]
 Thus, we obtain unique existence of solutions for \[
\dot{x}\left(t\right)=g\left(t,x\left(\alpha^{-1}\left(t\right)\right),\dot{x}\left(\beta^{-1}\left(t\right)\right)\right),\: t\in\mathbb{R}.\]
 \end{example}

\end{document}